\documentclass[11pt,reqno]{amsart}
\usepackage{fullpage}
\usepackage[mathscr]{eucal}
\usepackage{mathrsfs}
\usepackage{amsfonts}
\usepackage{amsmath}
\usepackage{amsthm}
\usepackage{amssymb}
\usepackage{latexsym}
\usepackage[all]{xy}
\usepackage{pstricks,pst-node}
\usepackage[mathscr]{eucal}

\theoremstyle{plain}
\newtheorem{defn}[equation]{Definition}
\newtheorem{cor}[equation]{Corollary}
\newtheorem{lem}[equation]{Lemma}
\newtheorem{prop}[equation]{Proposition}
\newtheorem{thm}[equation]{Theorem}

\theoremstyle{remark}

\newtheorem{rem}[equation]{Remark}

\newtheorem{examp}[equation]{Example}

\newtheorem{notation}[equation]{Notation}
\numberwithin{equation}{subsection}

\makeatletter
\renewcommand{\subsection}{\@startsection{subsection}{2}{0pt}{-3ex
plus -1ex minus -0.2ex}{-2mm plus -0pt minus
-2pt}{\normalfont\bfseries}} \makeatother

\newcommand{\scr}[1]{\mathscr{#1}}

\newcommand{\dis}{\displaystyle}

\newcommand{\beq}{\begin{equation}\label}
\newcommand{\eeq}{\end{equation}}

\newcommand{\hdot}{{\:\raisebox{2pt}{\text{\circle*{1.5}}}}}
\newcommand{\hd}{{\:\raisebox{5pt}{\mbox{$\bullet$}}}}
\newcommand{\idot}{{\:\raisebox{2pt}{\text{\circle*{1.5}}}}}

\newcommand{\erem}{\hfill$\lozenge$\end{rem}\vskip 3pt }

\newcommand{\dok}{\vskip 2pt $\dis\begin{array}{r} \hskip }
\newcommand{\edok}{\end{array}$\qed}
\DeclareMathOperator{\Spec}{\mathrm{Spec}}

\DeclareMathOperator{\Sym}{\mathrm{Sym}}
\DeclareMathOperator{\Lie}{\mathrm{Lie}}

\DeclareMathOperator{\Image}{\mathrm{Image}}
\DeclareMathOperator{\Ker}{\mathrm{Ker}}
\DeclareMathOperator{\Hom}{\mathrm{Hom}}
\DeclareMathOperator{\Proj}{\mathsf{Proj}}

\DeclareMathOperator{\Rep}{\mathrm{Rep}}
\DeclareMathOperator{\ad}{\mathrm{ad}}
\DeclareMathOperator{\Ad}{\mathrm{Ad}}
\DeclareMathOperator{\Hilb}{\mathrm{Hilb}}
\DeclareMathOperator{\End}{\mathrm{End}}

\newcommand{\go}{{\operatorname{good}}}

\newcommand{\iso}{{\,\stackrel{_\sim}{\to}\,}}

\newcommand{\Tr}{\oper{Tr}}

\newcommand{\vs}{\vskip 1pt }

\newcommand{\qq}{\hfill$\lozenge$\break }
\newcommand{\gm}{{\C^\times}}
\newcommand{\oper}{\operatorname}

\newcommand{\obar}{\overline }

\newcommand{\F}{{\mathbb F}}
\renewcommand{\L}{\Lambda }

\newcommand{\Id}{\text{Id}}

\renewcommand{\top}{_{\oper{top}}}
\newcommand{\supp}{\text{supp}\,}

\newcommand{\op}{{\text{op}}}

\newcommand{\dq}{{\overline{Q}}}
\newcommand{\fq}{{Q^\heartsuit}}
\newcommand{\ddq}{{\overline{Q^\heartsuit}}}
\newcommand{\bv}{{\mathbf v}}
\newcommand{\bw}{{\mathbf w}}
\newcommand{\bi}{{\mathbf i}}
\newcommand{\bj}{{\mathbf j}}
\newcommand{\bx}{{\mathbf x}}
\newcommand{\by}{{\mathbf y}}
\newcommand{\cg}{{\C\Gamma} }

\newcommand{\ee}{{\mathbf e}}

\renewcommand{\bigl}{\big(}
\renewcommand{\bigr}{\big)}
\newcommand{\BS}{{\mathbb S}}
\newcommand{\R}{{\mathcal R}}
\newcommand{\M}{{\mathcal M}}
\renewcommand{\th}{\theta }
\newcommand{\ii}{_{i\in I}}
\newcommand{\cm}{{\C^\times}}
\def\ccirc{{{}_{\,{}^{^\circ}}}}

\renewcommand{\o}{\otimes }
\newcommand{\bplus}{\mbox{$\bigoplus$}}

\newcommand{\cd}{\!\cdot\!}
\newcommand{\FI}{{\mathfrak I}}
\newcommand{\Ga}{\Gamma }

\newcommand{\x}{{\widetilde{X}}}
\newcommand{\odd}{\mathrm{odd}}
\newcommand{\fg}{\g}
\newcommand{\ssl}{{\mathfrak{s}\mathfrak{l}}}

\def\ss{{\mathcal{S}}}
\def\ts{{\widetilde{\mathcal{S}}}}
\def\fz{{\mathfrak{z}}}

\newcommand{\nn}{{\mathcal{N}}}

\def\bb{{\mathcal{B}}}

\newcommand{\G}{\Gamma }
\newcommand{\FF}{{\mathcal F}}

\renewcommand{\O}{\mathcal{O}}
\newcommand{\Om}{\Omega}

\newcommand{\BR}{{\mathbb R}}
\newcommand{\g}{\mathfrak{g}}
\newcommand{\n}{\mathfrak{n}}

\newcommand{\wh}{\widehat }

\newcommand{\V}{{\scr V}}

\renewcommand{\o}{\otimes }

\newcommand{\half}{\mbox{$\frac{1}{2}$}}
\newcommand{\inv}{^{-1}}
\newcommand{\vi}{${\en\sf {(i)}}\;$}
\newcommand{\vii}{${\;\sf {(ii)}}\;$}
\newcommand{\viii}{${\sf {(iii)}}\;$}
\newcommand{\iv}{${\sf {(iv)}}\;$}
\newcommand{\vv}{${\sf {(v)}}\;$}

\newcommand{\sset}{\subset}

\newcommand{\wt}{\widetilde }

\newcommand{\sminus}{\smallsetminus}
\newcommand{\mto}{\mapsto}
\newcommand{\into}{{}^{\,}\hookrightarrow^{\,}}
\newcommand{\too}{\,\longrightarrow\,}
\newcommand{\onto}{\twoheadrightarrow}

\DeclareMathOperator{\gr}{\mathrm{gr}}

\def\ip<#1,#2>{\left\langle#1,#2\right\rangle}
\def\sp<#1>{\left\langle#1\right\rangle}

\def\ip<#1,#2>{\left\langle#1,#2\right\rangle}

\def\npb{\noindent$\bullet\quad$\parbox[t]{145mm}}
\def\hp{\hphantom{x}}

\newcommand{\la}{\lambda }
\newcommand{\om}{\omega}
\newcommand{\be}{\beta }
\newcommand{\al}{{\alpha}}

\newcommand{\en}{{\enspace}}

\def\oo{{\mathcal O}}

\def\gl{{\mathfrak{g}\mathfrak{l}}}

\def\Z{{\mathbb Z}}
\def\C{{\mathbb C}}
\newcommand{\arr}{\overset{{\,}_\to}}

\begin{document}

\title{\large{Lectures on Nakajima's Quiver Varieties}}

\author{Victor Ginzburg}\vskip 20mm
\address{Department of Mathematics, University of Chicago, 
Chicago, IL 60637, USA}
\email{ginzburg@math.uchicago.edu}
\maketitle

\centerline{\large{The summer school}}
\centerline{\large{ "Geometric methods in
representation theory"}}
\centerline{\large{Grenoble, June 16 - July 4, 2008}}
\bigskip
\bigskip
\centerline{\sf Table of Contents}
\vskip -1mm

$\hspace{20mm}$ {\footnotesize \parbox[t]{115mm}{
\hp${}_{}$\!\hp0.{ $\;\,$} {\tt Outline}\newline
\hp1.{ $\;\,$} {\tt Moduli of representations of quivers}\newline
\hp2.{ $\;\,$} {\tt Framings}\newline
\hp3.{ $\;\,$} {\tt Hamiltonian reduction for representations of quivers}\newline
\hp4.{ $\;\,$} {\tt Nakajima varieties}\newline
\hp5.{ $\;\,$} {\tt Lie algebras and quiver varieties}
}}

\bigskip
\section{Outline}
\subsection{Introduction}\label{intro}
Nakajima's quiver varieties are certain smooth (not necessarily affine) complex
algebraic varieties associated with quivers. These varieties
have been used by Nakajima to give a geometric construction of universal
enveloping algebras of Kac-Moody Lie algebras
(as well as a 
construction of  quantized enveloping algebras for 
{\em affine} Lie algebras)
and of all irreducible integrable (e.g.,
finite dimensional) representations of those algebras.

A connection between quiver representations and Kac-Moody Lie algebras
has been first discovered by C. Ringel around 1990. 
Ringel produced a  construction of $U_q(\n)$, the
{\em positive part} of the quantized enveloping algebra
$U_q(\g)$ of a Kac-Moody Lie algebra $\g$, in terms
of a {\em Hall algebra} associated with an appropriate quiver.
Shortly afterwards,
G. Lusztig 
combined Ringel's ideas with the powerful technique of perverse
sheaves to construct a {\em canonical basis} of
 $U_q(\n)$, see \cite{L2}, \cite{L3}.

The main advantage of Nakajima's approach (as opposed to the earlier one
by Ringel and Lusztig) is that it yields a geometric construction
of the whole  algebra
$U(\g)$ rather than its positive part. At the same time,
it also provides
a  geometric construction of simple integrable $U(\g)$-modules.
Nakajima's approach also yields a similar construction 
of the algebra $U_q(L\g)$ and its simple integrable
 representations, where $L\g$ denotes the
loop Lie algebra associated to $\g$.
\footnote{Note however that, unlike the Ringel-Lusztig construction,
the approach used by Nakajima does {\em not} provide a construction
of the {\em quantized} enveloping algebra $U_q(\g)$
of the Lie algebra $\g$ itself. A similar situation holds in the
case of Hecke algebras, where the {\em affine} Hecke algebra
has a geometric interpretation in terms of equivariant $K$-theory,
see \cite{KL}, \cite{CG}, while the Hecke algebra of a finite
Weyl group does not seem to have such an interpretation.}

There are several steps involved in the
 definition of Nakajima's  quiver varieties. 
Given a quiver $Q$, one associates to it three other quivers,
$\fq,\ \dq,$ and $\ddq$, respectively. In terms of these quivers,
various steps of the construction of
 Nakajima varieties may be illustrated schematically
as follows 
{\renewcommand{\arraystretch}{1.3}
$$
\xymatrix{
&{\begin{tabular}{|c|}
\hline
\text{Framed representation}\\
\text{variety $\Rep \fq$}\\
\hline
\end{tabular}}\en\ar@{.>}[dr]&\\
{\begin{tabular}{|c|}
\hline
$\Rep Q$\\
\hline
\end{tabular}}
\en\ar@{.>}[ur]\ar@{.>}[dr]&&
\en
{\begin{tabular}{|c|}
\hline
\textsf{Nakajima variety $\M_{\la,\th}(\bv,\bw)\,$:}\\
\hline
\text{Hamiltonian reduction of}\\
$\Rep \ddq=T^*(\Rep \fq),$\\
\text{(= cotangent bundle of framed}\\
\text{representation variety of $Q$)}\\
\hline
\end{tabular}}\\
&\en {\begin{tabular}{|c|}
\hline
\text{Hamiltonian reduction}\\
\text{of $\Rep \dq=T^*(\Rep Q)$}\\
\hline
\end{tabular}}
\en\ar@{.>}[ur]&
}
$$

\subsection{Nakajima's  varieties and symplectic algebraic geometry}
Nakajima's  varieties also provide  an important large class of examples
of algebraic symplectic manifolds with extremely nice
properties and  rich structure,  interesting in 
their own right. To explain this, it is instructive to consider
a more general setting as follows.

Let $X$ be
a (possibly singular)  affine variety equipped with an
algebraic Poisson structure.
In algebraic terms, this means that  $\C[X]$,
the coordinate  ring  of $X$,
is equipped with a  Poisson
bracket $\{-,-\}$, that is, with a Lie bracket satisfying
the Leibniz identity.

Recall that any smooth symplectic algebraic manifold carries a
natural Poisson structure.

\begin{defn} Let  $X$ be
an irreducible affine normal  Poisson variety.
A resolution of singularities $\pi: \x\to X$ is called
a {\em symplectic resolution}
of $X$ provided $\x$ is 
a smooth complex algebraic symplectic manifold
(with algebraic symplectic 2-form) 
such that the  pull-back morphism
 $\pi^*: \C[X]\to \Gamma(\x, \oo_{\x})$ is a  Poisson algebra morphism.
\end{defn}

Below, we will be interested in the 
case where the variety $X$ 
is equipped, in addition, with
a
$\C^\times$-action that rescales the Poisson bracket
and contracts $X$ to
a (unique) fixed point $o\in X$. Equivalently,
this means that the coordinate ring of $X$
is equipped  with a {\em nonnegative} grading
$\C[X]=\bigoplus_{k\in\Z} \C^k[X]$ such that $\C^k[X]=0\,(\forall k<0),$ and
$\C^0[X]=\C$ and, in addition, there exists a (fixed) positive
integer $m>0$,  such that one has
$$\{\C^i[X],\C^j[X]\}\sset \C^{i+j-m}[X],
\quad\forall i,j\geq 0.
$$

In this situation, given a  symplectic resolution $\pi: \x\to X$,
we call  $\pi\inv(o)$, the fiber  of $\pi$ over the $\C^\times$-fixed point $o\in X,$
the {\em central fiber}.

Symplectic resolutions of a Poisson variety with a
contracting $\C^\times$-action as above enjoy   a number of very favorable properties:

\noindent\vi\en\,\parbox[t]{153mm}{
The map $\pi: \x\to X$ is automatically {\em semismall}
in the sense
of Goresky-MacPherson, i.e. one has $\dim(\x\times_X\x)=\dim X$, cf. [K1].}

 \noindent\vii\,\,\parbox[t]{153mm}{
 We have
a  Poisson algebra  {\em isomorphism}
 $\pi^*: \C[X]\iso \Gamma(\x,{\mathcal O}_\x)$, moreover,
 $H^i(\x,{\mathcal O}_\x)=0$ for all $i>0$.
 The $\C^\times$-action on $X$ admits a canonical lift to an algebraic
$\C^\times$-action on $\x$, see [K1].}

\noindent\viii\parbox[t]{153mm}{
The Poisson variety $X$ is a union of {\em finitely many}
symplectic leaves $X=\sqcup X_\al,$  [K4], and each symplectic
leaf $X_\al$ is a  locally closed smooth algebraic subvariety
of $X$,  [BG].}

\noindent\iv\parbox[t]{153mm}{
For any $x\in X$, we have
$H^{\odd}(\pi\inv(x),\C)=0,$ moreover, the cohomology
group $H^{2k}(\pi\inv(x),\C)$ has pure Hodge structure of type $(k,k)$,
 for any $k\geq 0,$ cf. [EV] and  [K3].}

\noindent\vv\,\,\parbox[t]{153mm}{Each fiber of $\pi$,
equipped with reduced scheme structure,
 is an isotropic subvariety of $\x$.
The central fiber $\pi\inv(o)$ is a homotopy retract of $\x$, in particular,
we have $H^\hdot(\x,\C)\cong H^\hdot(\pi\inv(o),\C)$.}
\medskip

The set $\x\times_X\x$ that appears in (i) may have several irreducible
 components
and the semismallness condition says that the dimension of any such
 component is $\leq\dim X$; in particular, the
diagonal $X\sset\x\times_X\x$ is
one such component of maximal dimension. 
To prove (i), write $\omega$ for the symplectic 2-form on $\x$,
and equip $\x\times\x$ with the 2-form 
$\Omega:=p_1^*\omega+ p_2^*(-\omega)$, where
$p_i:\x\times\x\to\x,\ i=1,2,$ denote the projections.
Then, $\Omega$ is a symplectic
form on  $\x\times\x$ and it is not difficult
to show that the restriction of
$\Omega$ to the (regular locus of the) subvariety 
$\x\times_X\x$ vanishes. 
The  inequality $\dim  \x\times_X\x\leq\dim X,$
hence  the semismallness of $\pi$,
follows from this.

Essential parts of statements (ii) and  (iv) 
are special cases of the following
more general result, to be proved in section \ref{proofwiz}
below.

\begin{prop}\label{wiz} Let $\pi: \ \wt X \to X$ be a proper
morphism, where $\wt X$ is a smooth symplectic algebraic
variety and $X$ is an affine variety. Then, one has

\vi $H^i(\x,{\mathcal O}_\x)=0$ for all $i>0$.

\vii Any fiber of $\pi$ is an isotropic subvariety.
\end{prop}

\begin{examp}[{\textbf{Slodowy slices}}]\label{Nil}
Let $\fg$ be a complex semisimple Lie algebra 
 and  $\langle e,h,f\rangle\sset\fg$
 an $\ssl_2$-triple for a nilpotent element $e\in\fg$.
Write $\fz_f$ for the centralizer of $f$ in $\fg$,
and $\nn$ for the {\em nilcone}, the subvariety
of all nilpotent elements of $\fg$.
Slodowy has shown that the intersection $\ss_e:=\nn\cap (e+\fz_f)$ is
 reduced, normal, and that there is  a  $\C^\times$-action on
$\ss_e$ that contracts $\ss_e$ to $e$, cf. eg. \cite{CG}, \S 3.7
for an exposition.

The  variety $\ss_e$ is called the
{\em Slodowy slice} for $e$.
 
Identify
$\fg$ with $\fg^*$ by means of the Killing form,
and  view $\ss_e$ as
a subvariety in $\fg^*$.
Then, the standard Kirillov-Kostant
 Poisson structure on $\fg^*$  induces
 a  Poisson structure on
$\ss_e$. The symplectic leaves in $\ss_e$ are
obtained by intersecting $e+\fz_f$
with the various  nilpotent conjugacy classes in $\fg$.

Let $\bb$ denote  the flag variety
for $\fg$ and let $T^*\bb$ be the  cotangent bundle on $\bb.$
There is a standard
resolution of singularities $\pi: T^*\bb\to \nn,$  the {\em Springer
resolution},
cf. eg. [CG, ch. 3].
It is known that $\ts_e:=\pi\inv(\ss_e)$
is a smooth submanifold in  $T^*\bb$
and
the  canonical
symplectic 2-form on the cotangent bundle
restricts to a nondegenerate, hence symplectic, 2-form on $\ts_e$.
Moreover, restricting $\pi$ to $\ts_e$ gives
a symplectic resolution $\pi_e: \ts_e\to\ss_e$, cf. \cite{Gi},
Proposition 2.1.2.
The central fiber  of that resolution
is $\pi_e\inv(e)=\bb_e$, the fixed point set of the natural
action of the element $e\in\fg$ on the flag variety $\bb$.

In the (somewhat degenerate) case $e=0$, we have $\ss_e=\nn$, and the corresponding
symplectic resolution reduces to the Springer resolution itself.
\end{examp}

\begin{examp}[{\textbf{Symplectic orbifolds}}]
\label{klein} Let $(V,\om)$ be 
 a finite dimensional symplectic
vector space and $\G\sset Sp(V,\om)$ a finite subgroup.
The orbifold
$X:=V/\G$ is an affine normal algebraic variety,
and the   symplectic structure on $V$ induces
a Poisson structure on $X$. Such a variety
$X$ may or may not have a symplectic resolution
$\x\to X$, in general. 
This holds, for instance, in the case of  {\em Kleinian singularities},
that is the case where $\G\sset SL_2(\C)$ and  $X:=\C^2/\G.$
Then, a  symplectic resolution $\pi: \x\to X$ does exist.
It is
the canonical
minimal resultion, see \cite{Kro}.

Recall that there is
a correspondence, the McKay
correspondence, between  the (conjugacy classes of) finite subgroups  $\G\sset SL_2(\C)$ and
Dynkin graphs of $\mathbf{A},\mathbf{D},\mathbf{E}$ types,
cf. \cite{CS},
and \S\ref{mksec} below. 
It turns out that $\C^2/\G$ is isomorphic, as a  Poisson variety,
to the Slodowy slice $\ss_e$, where 
$e$ is a {\em subregular} nilpotent in the
simple Lie algebra $\g$ associated with the
Dynking diagram of the corresponding type.

Another important example is the case where
$\G\sset GL({\mathfrak{h}})$ is  a complex reflection group
and $V:={\mathfrak{h}}\times{\mathfrak{h}}^*=T^*{\mathfrak{h}}$ is  the cotangent bundle of 
the vector space ${\mathfrak{h}}$ equipped with the 
canonical
 symplectic structure of the cotangent bundle.
We get a natural imbedding $\G\sset Sp(V)$.
One can show that,
among all irreducible finite Weyl groups $\G$, 
only those 
of types $\mathbf{A},\mathbf{B},$ and $\mathbf{C}$,
have the property that the
orbifold $({\mathfrak{h}}\times{\mathfrak{h}}^*)/\G$
admits a symplectic resolution, see  [GK], [Go].

In type  $\mathbf{A}$, 
we have
 $\G=S_n$,
the Symmetric group acting
diagonally on 
 $\C^n\times \C^n$ (two copies of
 the permutation representation).
Thus,
$(\C^n\times \C^n)/S_n=(\C^2)^n/S_n$
is the $n$-th symmetric power of the plane $\C^2$.
The orbifold $(\C^2)^n/S_n$ has  a natural resolution of singularities
$\pi: \Hilb^n(\C^2)\to  (\C^2)^n/S_n$,
where $\Hilb^n(\C^2)$ stands for the
Hilbert scheme of $n$ points in $\C^2$.
The map  $\pi$,
called Hilbert-Chow morphism,
 turns out to be a symplectic resolution, cf. [Na3], \S 1.4.
\end{examp}

\begin{examp}[{\textbf{Quiver varieties}}]\label{quiver} Let $Q$ be a finite quiver 
with vertex set $I$. Let $\bv,\bw\in\Z^I_{\geq 0}$ be
a  pair of dimension vectors.
Nakajima varieties provide, in many cases, examples of a symplectic
resolution of the form $\M_{\th}(\bv,\bw)\to\M_0(\bv,\bw).$
Here, $\th\in\BR^I$ is a `stability parameter',
and we write  $\M_{\th}(\bv,\bw)$ for the  Nakajima variety
$\M_{0,\th}(\bv,\bw)$, as defined in Definition \ref{def_M} of \S\ref{sec_M}
below. For $\th=0$,
the  variety
$\M_{\th}(\bv,\bw)$  is known to be  affine, see Theorem \ref{Mthm}(i).

Assume now that $\th$ is
 chosen to lie outside a certain collection $\{H_j\}$
of `root hyperplanes' in $\BR^I$.
Then, under fairly mild conditions, the  Nakajima variety $\M_{\th}(\bv,\bw)$
turns out to be a smooth  algebraic variety that
comes equipped with a natural hyper-K\"ahler structure.
The (algebraic)
symplectic structure on $\M_{\th}(\bv,\bw)$ is a part
of that 
 hyper-K\"ahler structure. This part is independent of the
choice of the stability parameter  $\th$ 
as long as $\th$ stays within a connected component
of the set $\BR^I\sminus (\cup_j\ H_j)$.
In contrast, the K\"ahler structure on  $\M_{\th}(\bv,\bw)$
does depend on the choice of $\th$ in an essential way.

Nakajima's varieties encorporate many of the examples described above.
For a simple Lie algebra of type $\mathbf A$, for instance,
{\em all} symplectic resolutions described in Example \ref{Nil}
come from appropriate quiver varieties, see \cite{Ma}.

Similarly, the minimal resolution of a Kleinian singularity
and the resolution $\pi: \Hilb^n(\C^2)\to  (\C^2)^n/S_n$,
see Example \ref{klein}, are also special cases of 
symplectic resolutions arising
from quiver varieties. There are other important 
examples as well, eg. the ones where the group $\G$ is a wreath-product.

Quiver varieties provide a unifying framework for all these examples,
from both conceptual and technical
 points of view. Here is an illustration of this.

\begin{rem} The odd cohomology vanishing for the fibers of
the Springer resolution, equivalently, for the
$e$-fixed point varieties  $\bb_e\sset\bb$,
was standing as an open problem for
quite a long time. This problem  has been finally solved in [DCLP].
The argument in  [DCLP] is quite complicated, in particular, it
involves a case-by-case analysis.

 The odd cohomology vanishing for the fibers of
the map  $\M_{\th}(\bv,\bw)\to\M_0(\bv,\bw)$ was proved in \cite{Na4}.
Nakajima's proof is based on
 a standard result saying that rational homology groups of a
complete variety that admits a `resolution of diagonal' in $K$-theory,
cf. \cite[Theorem 5.6.1]{CG}, is spanned by the
fundamental classes of algebraic cycles.\footnote{It is not
known whether it is true or not that,
for any nilpotent element $e$ in an arbitrary semisimple
Lie algebra $\g$,
 the variety $\ts_e$, cf. Example \ref{Nil},
 admits a resolution of diagonal in $K$-theory.}

Property
(iv) of symplectic resolutions stated earlier in this subsection provides an alternative,
 more conceptual, unified approach to the odd
cohomology vanishing of the fibers of
the map $\pi$ in the above examples.
\end{rem}
\end{examp}

\subsection{Reminder}\label{reminder} Throughout the paper, the ground field is
 the field $\C$ of complex numbers. 

We fix a quiver $Q$,
i.e., a finite oriented graph, with vertex set $I$ and edge
set $E$. We write $Q^\op$ for the  opposite quiver
 obtained from $Q$ by
reversing the orientation of  edges.

For any pair $i,j\in I$, let
 $a_{ij}$ denote the number of edges of $Q$ going from $j$ to $i$.
The matrix $A_Q:=\|a_{ij}\|$ is called the {\em adjacency matrix}
of $Q$. 

On $\C^I$, one has the standard euclidean inner product
$\al\cdot\be:=\sum\ii\ \al_i\be_i.$ 
Thus, the (non-symmetric) bilinear form associated with the adjacency matrix reads
$$A_Q\al\cd \be=\sum_{x\in E}\
\al_{\text{tail}(x)}\,\be_{\text{head}(x)},\qquad  \al,\be\in\C^I.
$$

Let $\C I$ be the algebra of $\C$-valued
functions on the set $I$, equipped with pointwise
multiplication. This is a finite dimensional
semisimple commutative algebra isomorphic to
$\oplus\ii \C$. We write $1_i$ for the
characteristic function of the one point set $\{i\}\sset I$.

Let $\C E$ be a $\C$-vector space with basis $E$.
The vector space $\C E$ has a natural $\C I$-bimodule structure
such that, for any edge $x\in E$, we have $1_j\cdot x\cdot 1_i=x$
if $j=\text{tail}(x)$ and $i=\text{head}(x)$, and
$1_j\cdot x\cdot 1_i=0$ otherwise.

One defines
the {\em path algebra of $Q$} as  $\C Q:=T_{\C I}(\C E)$,
the tensor algebra of the $\C I$-bimodule $\C E$.
For each $i\in I$, the element
$1_i\in\C I\sset \C Q$ may be identified with the trivial
path at the vertex $i$.

Let $B$ be an arbitrary
 $\C$-algebra  equipped with an algebra map
$\C I\to B$, eg. $B$ is a quotient of the path algebra of a quiver.
Abusing the notation, we also write $1_i$ for the image of 
the element $1_i\in \C I$ in $B$.
 Associated with any
finite dimensional left $B$-module $M$, there is its  {\em dimension vector}
$\,\dim_I M\in \Z^I_{\geq 0},$ such that the $i$th
coordinate of $\dim_I M$ equals
$(\dim_I M)_i:=\dim(1_i\cdot M),$ where we always write $\dim=\dim_\C$.

Note that a left $\C I$-module is the same thing as
 an $I$-graded vector space. Given an $I$-graded finite-dimensional vector space
$V=\oplus\ii V_i$, we let $\Rep(B,V)$ denote
the set of algebra homomorphisms $\rho: B\to \End_\C V$ such that
$\rho|_{\C I}$, the pull-back of $\rho$ to the subalgebra $\C I$,
equals the   homomorphism coming from the $\C I$-module structure on $V$.
The group $\prod\ii GL(V_i)$ acts naturally on
$\Rep(B,V)$ by `base change' automorphisms.

Let $\bv=(v_i)_{i\in I}\in\Z^I_{\geq 0}$ be an $I$-tuple,
to be referred to as  a `dimension vector'.
Given an $I$-graded vector space $V=\oplus\ii V_i$, 
such that $\dim V_i=v_i$ for all $i\in I$,
we will often abuse the notation
and write $GL(v_i)$ for $GL(V_i),$ resp.
 $\Rep(B,\bv)$ for $\Rep(B,V)$. 
In the special case $B=\C Q$, we 
simplify the notation  further and write
$\Rep(Q,\bv):=\Rep(\C Q,\bv)$, the
space of $\bv$-dimensional representations of $Q$.

We put
$G_\bv:=\prod_{i\in I}\ GL(v_i).$
Thus, $G_\bv$ is a reductive group, and
$\Rep(Q,\bv)$ is a vector space
that comes equipped with a linear
$G_\bv$-action, by base change
automorphisms.
We have
\beq{dims}
\dim\Rep(Q,\bv)=A_Q\bv\cdot\bv,
\qquad\dim G_\bv=\bv\cdot\bv.
\eeq

Note the the subgroup
$\cm\sset G_\bv,$ of {\em diagonally imbedded}
invertible scalar matrices
acts trivially on $\Rep(Q,\bv)$.

We will very often use the following elementary result

\begin{lem}\label{connected} Let $B$ be an algebra equipped with an
algebra map $\C I \to B$.
Then, the isotropy group of any point of $\Rep(B,\bv)$
is a {\em connected} subgroup of $G_\bv$.
\end{lem}
\begin{proof} Let $M$
be a representation of $B$, and write $\End_BM$
for the algebra of $B$-module endomorphisms of
$M$. It is known (and easy to see) that the isotropy group $G^M$ of
the point  $M\in \Rep(B,\bv)$ may be identified with
the group of invertible elements of the algebra  $\End_BM$.

We claim that, more generally, the set $A^\times$ of 
 of invertible elements of any finite dimensional $\C$-algebra
$A$ is connected. To see this, we observe
that the set $A^\text{sing},$
of noninvertible
elements of $A$, is a hypersurface in $A$ given by the
equation $\det m_a=0,$ where $m_a$ denotes the operator
of left multiplication by an element $a\in A$.

Such a hypersurface has real  codimension $\geq 2$ in $A$,
hence cannot disconnect $A$,  a real vector space.
Therefore, the
set $A^\times=A\sminus A^\text{sing}$, of invertible elements,
 must be connected.
\end{proof}

\section{Moduli of representations of quivers}\label{int1} 
\subsection{Categorical quotients}
Naively, one would like
to consider a space of isomorphism classes
of representations of $Q$ of some fixed dimension $\bv$.
Geometrically, this amounts to considering the orbit space
$\Rep(Q,\bv)/G_\bv$. Such an  orbit space is, however,
rather `badly behaved' in most cases. Usually, it does not have a reasonable 
Hausdorff topology, for instance.

One way to define a reasonable orbit space
is to use a {\em categorical quotient} 
$$\Rep(Q,\bv)/\!/G_\bv:=\Spec \C[\Rep(Q,\bv)]^{G_\bv},$$
the spectrum of the algebra of $G_\bv$-invariant polynomials
on the vector space $\Rep(Q,\bv).$ By definition,
$\Rep(Q,\bv)/\!/G_\bv$ is an affine algebraic variety.

To understand the categorical quotient,
we recall the following result of Le Bruyn and Procesi, \cite{LBP},

\begin{prop}\label{BP} The algebra $\C[\Rep(Q,\bv)]^{G_\bv}$ 
is generated by the functions
$\Tr(p,-) :\  V\mto \Tr(p, V)$, where
$p$ runs over the set of oriented cycles in $Q$ of the form
$p= p_{i_1,i_2}\cdot p_{i_2,i_3}\cdot\ldots\cdot
p_{i_{s-1},i_s}\cdot p_{i_s,i_1},\ (p_{ij}\in E),$
and where $\Tr(p,V)$ denotes the trace of the
operator corresponding to such a cycle in the 
representation $V\in \Rep(Q,\bv)$.
\end{prop}

The above proposition is a simple consequence of
H. Weyl's `{\em first fundamental theorem of Invariant theory}',
cf. \cite{Kra}. The proposition yields

\begin{cor}\label{pt} For a quiver $Q$ without oriented cycles,
one has $\C[\Rep(Q,\bv)]^{G_\bv}=\C$, hence, we have
$\Rep(Q,\bv)/\!/G_\bv=pt$.\qed
\end{cor}

Combining Proposition \ref{BP} with standard results from
invariant theory, cf. \cite[Theorem 5.9]{Mu},
one obtains the following

\begin{thm}\label{ssreps} Geometric (= closed) points of the scheme
$\Spec\C[\Rep(Q,\bv)]^{G_\bv}$ are in a natural bijection 
with $G_\bv$-orbits of semisimple representations of $Q$.\qed
\end{thm}

 Corrolary \ref{pt} shows that
the categorical quotient may often reduce to a point,
so a lot of geometric information may be lost.

A better approach to the moduli problem is to use a
stability condition and to replace the orbit space
$\Rep(Q,\bv)/G_\bv$,
 or  the categorical quotient
$\Rep(Q,\bv)/\!/G_\bv$, by an appropriate moduli space of
(semi)-stable representations. There is a price to pay:
moduli spaces arising in this way do depend on the
choice of a stability condition, in general.

\subsection{Reminder on GIT}
 The general theory of quotients by a reductive group
action via stability conditions has been developed
by D. Mumford, and is called Geometric Invariant Theory, cf.
\cite{GIT}.

To fix ideas,
let $X$ be a  not necessarily irreducible,
 affine algebraic $G$-variety,
where $G$ is a reductive algebraic group.
Given a rational character  (= algebraic group
homomorphism)  $\chi: G\to\gm$,
Mumford defines a scheme $X/\!/_\chi G$ in the following
way.  
Let $G$ act on the cartesian
product $X\times \C$ by the formula
$g:\ (x,z)\mto (gx, \chi(g)\inv\cdot z)$ (more generally, the cartesian
product $X\times \C$  may be replaced  here by the total space of any
$G$-equivariant line bundle on
$X$). The coordinate ring of
$X\times \C$ is the algebra
 $\C[X\times \C]=\C[X][z]$, of polynomials in the variable $z$ with
 coefficients in the coordinate ring of
$X$. This algebra has an obvious grading by degree
of the polynomial.

Let $A_\chi:=\C[X\times \C]^{G}$ be the subalgebra 
$G$-invariants. Clearly, this 
is a graded subalgebra which is, moreover,
a finitely generated algebra by Hilbert's theorem on finite
generation of algebras of invariants,
cf. \cite[ch. II, \S3.1]{Kra}.
Explicitly, a polynomial
$f(z)=\sum_{n=0}^N f_n\cdot z^n\in \C[X][z]$
is $G$-invariant if and only if, for each $n=0,\ldots,N$,
the function $f_n$ is a $\chi^n$-{\em semi-invariant},
i.e. if and only if  one has
$$
f_n(g\inv(x))=\chi(g)^n\cdot f_n(x),\quad \forall
g\in G,\  x\in X.
$$

Write $\chi^n:\ g\mto \chi(g)^n$
 for the $n$-th power of the character $\chi$ and let $\C[X]^{\chi^n}
\sset \C[X]$
be the vector space of $\chi^n$-semi-invariant functions.
It is clear
that we have
$$
A_\chi:=\C[X\times \C]^{G}=\bplus_{n\geq 0}\ \C[X]^{\chi^n},$$
and the direct sum decomposition on the right corresponds to the
grading on the algebra ~$A_\chi$.

Let $X/\!/_\chi G:=\Proj A_\chi$ be the projective spectrum of the 
graded algebra $A_\chi$. This is a  quasi-projective scheme,
called a {\em GIT quotient} of $X$ by the $G$-action; the scheme
$X/\!/_\chi G$ is reduced,
resp. irreducible,
whenever $X$ is  (since $
A_\chi$  has no nilpotents, resp. no zero divisors, provided there are
no nilpotents resp. no zero divisors, in
$\C[X]$).

Put $A_\chi^{>0}:=\bplus_{n> 0}\ \C[X]^{\chi^n}$.
Let $\scr I$ be the set of  homogeneous ideals $I\sset A_\chi$  such
that one has $I\neq  A_\chi$ and   $A_\chi^{>0}\not\subset I$.
An ideal  $I\in \scr I$ is said to be
a `maximal  homogeneous ideal' if it is not properly contained 
in any other ideal $I'\in \scr I$.
Geometric points of the scheme  $X/\!/_\chi G$ correspond to the maximal  homogeneous ideals.

In general,  for $n=0$, we have
$\C[X]^{\chi^n}=\C[X]^{G}$,
is the algebra of $G$-invariants. 
Thus, we have a canonical
algebra imbedding $\C[X]^{G}\into
A_\chi$ as the degree zero subalgebra.
Put another way, the algebra imbedding
$\C[X]^{G}\into\C[X\times \C]^{G}=A_\chi$ is
induced by the first projection $X\times \C\to X$.

Standard results of algebraic geometry imply that
the algebra  imbedding  $\C[X]^{G}\into
A_\chi$ induces
 a {\em projective} morphism of schemes
 $\pi:\ \Proj A_\chi \to \Spec\C[X]^{G}= X/\!/G.$

\begin{rem} In the special case where $G=\gm$ and
$A=\C[u_0,u_1,\ldots,u_m]$, is a polynomial algebra,
we have
$\Proj A={\mathbb P}^m=
(\C^{m+1}\sminus\{0\})/\gm$. 

More generally, given a reductive group $G$ and a {\em nontrivial}
character $\chi: G\to\gm$, put $K:=\Ker\chi.$ Thus, $K$ is a normal
subgroup of $G$ and one has $G/K=\gm$. 

Now, let $X$ be an affine $G$-variety
such that $\C[X]^{\chi^n}=0$ for any $n<0$.
Let $X/\!/K=\Spec(\C[X]^K)$ be the categorical quotient of $X$ by the $K$-action.
There is  a natural residual action of the group
$G/K=\gm$ on  $X/\!/K$,
equivalently, there is a natural nonnegative 
grading on the algebra  $\C[X]^K$.
Then, it is straightforward to show that $X/\!/_\chi G\cong\Proj (\C[X]^K)$.
Furthermore, geometric points of the  scheme $\Proj (\C[X]^K)$
correspond
to $\gm$-orbits in $(X/\!/K)\sminus Y$,
where $Y$ denotes the  set of $\gm$-fixed points in  $X/\!/K$.
\erem

\begin{rem}\label{veronese} For any character
$G\to\gm$ and any positive integer
$m>0$, one may view the algebra $A_{m\chi}$ as a graded subalgebra in $A_\chi$ via
the  natural imbedding
$A_{m\chi}=\bplus_{n\geq 0,\ m|n}\ \C[X]^{\chi^{n}}\into
A_\chi=\bplus_{n\geq 0}\ \C[X]^{\chi^{n}},$ called
the {\em Veronese imbedding}. One can show that the Veronese imbedding
induces an isomorphism $X/\!/_\chi G\iso X/\!/_{m\chi} G$, of algebraic
varieties.
\erem

Given a nonzero homogeneous semi-invariant $f\in A_\chi$ we put
$X_f:=\{x\in X\mid f(x)\neq0\}.$
To get a better understanding of
the GIT quotient $X/\!/_\chi G,$ one introduces the
following definition, see \cite{GIT}.

\begin{defn}\label{def_stab} \vi A point $x\in X$ is called {\em $\chi$-semistable} if
there exists $n \geq1$ and a $\chi^n$-semi-invariant
$f\in \C[X]^{\chi^n}$ such that $x\in X_f$.\vs

\vii A  point  $x\in X$ is called {\em $\chi$-stable} if
there exists $n\geq 1$ and a $\chi^n$-semi-invariant
$f\in \C[X]^{\chi^n}$ such that $x\in X_f$ and, in addition,
one has: (1)
the action map $G\times X_f\to X_f$ is a closed morphism
and (2) the isotropy group of the point $x$ is  finite.\vs

Write $X^{ss}_\chi,$ resp. $X^s_\chi$, for the set of semistable, resp.
stable, points. Thus, we have $X^s_\chi\sset X^{ss}_\chi\sset ~X$.\vs
\vs

\viii Two  $\chi$-semistable points $x,x'$ are called
{\em $S$-equivalent} if and only if  the orbit closures
$\overline{G\cdot x}$ and $\overline{G\cdot x'}$
meet in $X^{ss}_\chi$.
\end{defn}

Note that the $G$-orbit of a stable point is an orbit of maximal
dimension,
equal to $\dim G$, moreover, such a stable orbit is closed 
in $X^{ss}_\chi$. Hence, two
stable points are $S$-equivalent if and only if  they belong to the same orbit.

By definition, we have that $X^{ss}_\chi=\cup_{\{f,\,\deg f>0\}}\ X_f$
is a $G$-stable Zariski open subset of $X$.
Furthermore, there is a well defined morphism
$F:\ X^{ss}_\chi\to X/\!/_\chi G$, of algebraic varieties,
which is constant on $G$-orbits. 
The image of a $G$-orbit $\O\sset X^{ss}_\chi$
is a point corresponding to the maximal homogeneous ideal  $\FI_\O\sset A_\chi$
formed by the functions
 $f\in A_\chi$ such that $f(\O)=0$.

One of the  basic results of GIT reads

\begin{thm}\label{GIT}
\vi The morphism $F$ 
induces  a natural bijection between the set of $S$-equivalence classes of
$G$-orbits in $X^{ss}_\chi$ and
 the set of geometric
points of the scheme  $X/\!/_\chi G$.

\vii The image of the set of stable points
is  a Zariski open
(possibly empty) subset $F(X^s_\chi)\sset X/\!/_\chi G$;
moreover, the fibers of the restriction  $F: \ X^s_\chi\to X/\!/_\chi G$
are closed $G$-orbits of maximal dimension, equal to $\dim G$.
\end{thm}

\begin{examp}\label{triv} For the trivial character $\chi=1$, 
we have $A_\chi=\C[X]^G[z]$. The regular function $z\in A_\chi$ is
a homogeneous degree one  regular function that does not vanish
on $X$. Therefore, we have $X=X_z$ and any point $x\in X$ is  $\chi$-stable.
Such a point is $\chi$-stable  if and only if  the $G$-orbit
of $x$ is a closed orbit in $X$ of dimension $\dim G$. 
Furthermore, one has
$$X/\!/_\chi G\ =\ \Proj A_\chi\ =\ \Proj \big(\C[X]^G[z]\big)\
 =\ \Spec \C[X]^G\ =\ X/\!/G, \en\text{for}\en\chi=1.$$

In this case,
the canonical map $\pi$ becomes an isomorphism $X/\!/_\chi G\iso X/\!/G$.
\end{examp}

We will frequently use the following result which is, essentially, a
consequence of definitions.
\begin{cor}\label{sss} \vi
Let $X$ be a smooth $G$-variety such that the isotropy group
of any point of $X$ is connected. Then the set $F(X^s)$
is contained in the smooth locus of the scheme  $X/\!/_\chi G$.

\vii Assume, in addition, that $X$ is affine and that
the $G$-action on $X^{ss}$ is free.  Then
any semistable point is stable, the scheme  $X/\!/_\chi G$
is smooth. Furthermore,  the morphism
$F: X^{ss}\to X/\!/_\chi G$ is a principal $G$-bundle (in \`etale topology).\qed
\end{cor}
 
In the situation of (ii) above, one often calls the map $F$ (or the
 variety $X/\!/_\chi G$)
 a {\em universal geometric quotient}.

\subsection{Stability conditions for quivers}\label{king_sec} 
A. King introduced a totally different, purely algebraic, notion of stability 
for representations of algebras. He then showed that, in the
case of quiver representations, his definition of stability is actually equivalent
to Mumford's Definition \ref{def_stab}.

 To explain 
 King's approach, fix a quiver $Q$ and  fix $\th\in\BR^I$.
It will become clear shortly that the parameter  $\th$ is
an analogue of the group character
$\chi: G \to \gm$, in Mumford's theory.

Let  $V=\oplus\ii V_i$ be a finite dimensional nonzero representation of $Q$ 
 with dimension vector $\dim_I V\in \Z^I$.
One defines the {\em slope} of $V$
by the formula ${\operatorname{\textsl{slope}}}_\th(V):=(\th\cdot\dim_I V)/\dim_\C V,$
where  $\dim_\C V:=\sum\ii \dim V_i$.
Using the vector $\th^+:=(1,1,\ldots,1)\in \Z^I$,
one can alternatively write $\dim_\C V=\th^+\cdot\dim_I V.$

\begin{defn}\label{stab}  A nonzero representation $V$ of $Q$ is said
to be  $\th$-semistable 
if, for any subrepresentation $N\subset V$, we have
${\operatorname{\textsl{slope}}}_\th(N)\leq {\operatorname{\textsl{slope}}}_\th(V)$.

A  nonzero representation  is called  $\th$-stable
 if the strict inequality holds   for any
nonzero proper subrepresentation
 $N\subset V$.
\end{defn}

\begin{examp} Let $\th=0$. Then, any representation is 
 $\th$-semistable. Such a representation is
 $\th$-stable if and only if it is simple as an $\C Q$-module.
\end{examp}

\begin{rem} \vi Our definition of semistability in terms of
slopes  follows the approach
of Rudakov \cite[\S 3]{Ru}. In the case where  $\th\cdot\dim_I V=0$
 the
inequality of slopes in  Definition \ref{stab}
 reduces to the condition that
$\th\cdot N\leq 0$.
The latter condition, combined with
the requirement that $\th\cdot\dim_I V=0,$ is the original
definition of semistability
used by King \cite{Ki}. Rudakov's  approach is more flexible since it works
well without the assumptions that  $\th\cdot\dim_I V=0.$

\vii Let  $\th\in\BR^I$ and put $\th'=\th-c\cdot\th^+$, where $c\in\BR$ is an
arbitrary constant.
It is easy to see that  a representation $V$
is $\th$-semistable in the sense of  Definition \ref{stab} if and only if it is
 $\th'$-semistable. On the other hand,
given $V$, one can always find a constant  $c\in\BR$ such that
one has $\th'\cdot\dim_I V=0$, see \cite{Ru}, Lemma 3.4.
\erem

The definition of (semi)stability  given above is a special case
of a more general approach due to A. King \cite{Ki}, who
 considers the case of an arbitrary associative $\C$-algebra
$A$. 

Given such an algebra $A$,
let $K_{\text{fin}}(A)$ denote the Grothendieck group
of the category of {\em finite dimensional}
$A$-modules. This is a free abelian group
with the basis formed by the classes of
simple finite dimensional
$A$-modules.
Note that
the assignment $V\mto \dim_\C V$
extends to  an additive group homomorphism
$K_{\text{fin}}(A)\to\BR$.

Given any other  additive group homomorphism
$\phi: \ K_{\text{fin}}(A)\to\BR$
and a nonzero  finite dimensional $A$-module $V$,
one puts ${\operatorname{\textsl{slope}}}_\phi(V):=\phi([V])/\dim_\C V,$
where $[V]$ stands for the class of $V$ in $K_{\text{fin}}(A)$.
Following King and Rudakov, one says that a finite dimensional
$A$-module $V$ is   $\phi$-semistable, 
if for any nonzero $A$-submodule $N\sset V$, we have
${\operatorname{\textsl{slope}}}_\phi(N)\leq {\operatorname{\textsl{slope}}}_\phi(V)$. 

This definition specializes to Definition \ref{stab}
as follows. One takes $A:=\C Q$. 
Then, the assignment $[V]\mto \dim_I V$ yields
a well defined group homomorphism $\dim_I :\
K_{\text{fin}}(\C Q)\to \Z^I$. Now, for any $\th\in \BR^I$,
define a group homomorphism $\phi_\th: \ \Z^I\to\BR,\
x\mto \sum_i\ \th_i\cdot x_i.$  This yields an obvious isomorphism
$\BR^I\iso\Hom(\Z^I,\BR),\ \th\mto\phi_\th$. Thus, given 
$\th\in \BR^I$, one may form
a composite homomorphism
$K_{\text{fin}}(\C Q)\to \Z^I\to\BR,$
$[V]\mto \th\cdot\dim_I V$.

For this last  homomorphism,  the general definition of semistability
for $A$-modules
reduces to  Definition \ref{stab}.

\begin{rem} Assume that the quiver $Q$ has no oriented cycles.
Then, it is easy to show that any
simple representation $V$ of $Q$ is 1-dimensional,
i.e., there exists a vertex $i\in I$ such that
$V_i=\C$ and $V_j=0$ for any $j\neq i$.
It follows  that the map $\dim_I:\ K_{\text{fin}}(\C Q)\to \Z^I,\
[V]\mto\dim_I V$
is in this case a group {\en isomorphism}.
\erem

\begin{prop}\label{ab} Fix an additive group homomorphism
$\phi: \ K_{\text{fin}}(A)\to\BR$. Then,
finite dimensional $\phi$-semistable $A$-modules form an abelian
category. An $A$-module is $\phi$-stable if and only if it is a
simple object of this category. 
\end{prop}
\begin{proof} 
The proposition states that,
for any pair $M,N,$ of $\phi$-semistable $A$-modules,
the kernel, resp. cokernel, of an
$A$-module map $f: M\to N$ is  $\phi$-semistable again.

To prove this, put $K:=\Ker(f)$ and write
${\operatorname{\textsl{slope}}}(-)$ for
${\operatorname{\textsl{slope}}}_\th(-)$.
Then,
the imbedding $M/K\into N$ yields
${\operatorname{\textsl{slope}}}(M/K)\leq
{\operatorname{\textsl{slope}}}(N)={\operatorname{\textsl{slope}}}(M)$. Hence,
alying \cite[Lemma 3.2 and Definition 1.1]{Ru} to the short exact
sequence $K\to M\to M/K$, we get 
${\operatorname{\textsl{slope}}}(K)\geq {\operatorname{\textsl{slope}}}(M)$.
On the other hand, since $K$ is a submodule of $M$, a semistable module,
we have ${\operatorname{\textsl{slope}}}(K)\leq
{\operatorname{\textsl{slope}}}(M)$.
 Thus, we obtain that
${\operatorname{\textsl{slope}}}(K)={\operatorname{\textsl{slope}}}(M)$.
It follows that,
 for any submodule $E\sset K$,
we have
${\operatorname{\textsl{slope}}}(E)\leq{\operatorname{\textsl{slope}}}(M)=
{\operatorname{\textsl{slope}}}(K)$, since $E\sset M$.
Thus, we have proved that $K$ is $\phi$-semistable.

Next, write $C$ for the cokernel of the map $f$.
Then, one proves that ${\operatorname{\textsl{slope}}}(C)={\operatorname{\textsl{slope}}}(N)$
and, moreover, ${\operatorname{\textsl{slope}}}(C)\leq
{\operatorname{\textsl{slope}}}(F)$ 
for any quotient
$F$ of $C$. This implies that $C$ is a $\phi$-semistable $A$-module,
see \cite{Ru}, Definition 1.6 and the discussion after it.

We leave details to the reader.\end{proof}

We return to the quiver setting, and fix a quiver $Q$ with vertex set $I$.
As a corollary of Proposition \ref{ab}, we deduce that
any  $\th$-semistable representation $V$, of $Q$,
has a  Jordan-H\"older filtration $0=V_0\sset V_1\sset\ldots\sset V_m=V,$
by subrepresentations, such that $V_k/V_{k-1}$ is a
 $\th$-stable representation for any $k=1,\ldots, m.$
The associated graded representation $\gr^s V:=\oplus_k\ V_k/V_{k-1}$
does not depend, up to isomorphism, on the choice
of such a filtration.

To relate Mumford's and King's notions of stability,
we 
associate with an integral vector $\th=(\th_i)\ii\in\Z^I,$
a rational character
$$\chi_\th:\ G_\bv\to\cm, \quad
g=(g_i)\ii\ \mto\
\prod\ii \det(g_i)^{-\th_i}.
$$

\begin{rem} Fix  a representation $V$, of $Q$, and put $\bv=\dim_I V$.
It is clear that  the character $\chi_\th$ vanishes on the subgroup
 $\cm\sset G_\bv$ if and only if we have
$\th\cdot\bv=0$.
\erem

The main result of King relating the two notions of stability 
 reads

\begin{thm}\label{kingthm} For any dimension vector $\bv$ and any  $\th\in \Z^I$ such
that $\th\cdot\bv=0$, we have

\vi A representation $V\in \Rep(Q, \bv)$ is $\chi_\th$-semistable,
resp. $\chi_\th$-stable, in the sense of Definition \ref{def_stab}
if and only if  it is  $\th$-semistable,
resp. $\th$-stable, in the sense of Definition ~\ref{stab}.

\vii A pair $V,V'$, of   $\chi_\th$-semistable
representations, are $S$-{\em equivalent}
in the sense of Definition \ref{def_stab}
if and only if  one has $\gr^s V\cong \gr^s V'$.\qed
\end{thm}

Let $\Rep^s_\th(Q,\bv)$ denote the set of stable,
resp. $\Rep^{ss}(Q,\bv)$ denote the set of semistable,
representations of dimension $\bv$. 
We write
$\R_\th(\bv)=\R_\th(Q,\bv):=\Rep^{ss}_\th(Q,\bv)/\!/_{\chi_\th}G_\bv$.
  By Theorem \ref{GIT}, this is a quasi-projective
variety.

\begin{cor}[A. King]\label{stabrep} \vi
The group $G_\bv/\gm$ acts freely on the set $\Rep^s_\th(Q,\bv)$,
of  $\th$-{\em stable} representations.
The orbit set $\R^s_\th(\bv):=\Rep^s_\th(Q,\bv)/G_\bv$ 
is contained in $\R_\th(\bv)$ as a Zariski open (possibly empty)
subset.\vs

\vii Assume that $Q$ has no edge loops.
Then, the vector $\bv\in\Z^I_{\geq 0}$ is a {\em Schur vector}
(i.e. there exists a simple representation of $Q$ of dimension $\bv$)
for $Q$ if and only if  there exists $\th\in \Z^I$ such that  
$$\th\cdot\bv=0\quad\oper{and}\quad
\Rep^s_\th(Q,\bv)\neq\emptyset.$$

For such a $\th$, we have
$\dis\ \dim\R^s_\th(\bv)=1+A_Q\bv,\bv-\bv\cdot\bv.$
\end{cor}
\begin{proof}[Proof of (i)] Let
$g$ be an element of the isotropy group of $V$,
such that $g\notin\gm$, and let $c\in\C$ be an eigenvalue of 
$g$. Then $N:=\Ker(g-c\Id)$ is a nontrivial subrepresentation
of $V$. Clearly, the group $\gm$ acts trivially on $N$.
Hence, we have  $\dim_I N\cdot\bv=0,$  contradicting
the definition of stability. It follows
that the group $G_\bv/\gm$ acts freely  on  $\Rep^s_\th(Q,\bv)$.
\end{proof}

According to  Example \ref{triv}, we get
\begin{cor}\label{trivcor} In the special case $\th=0$, one has
$$\R_0(\bv)\ =\ \Rep(Q,\bv)/\!/G_\bv\ =\ \Spec\C[\Rep(Q,\bv)]^{G_\bv}.$$

For any $\th\in\Z^I$, there is a canonical projective morphism
$\pi:\
\R_\th(Q,\bv)\to \Rep(Q,\bv)/\!/G_\bv$. \qed
\end{cor}


\begin{rem}\label{opp}
Note that a representation $V$ of $Q$ is
 $\th$-semistable if and only if  $V^*$, the dual representation of $Q^\op$,
is  $(-\th)$-semistable. Thus, taking the dual representation yields
canonical isomorphisms
$$\Rep^{ss}_\th(Q,\bv)\iso\Rep^{ss}_{-\th}(Q^\op,\bv),\quad\text{resp.}
\quad
\R_\th(Q,\bv)\iso\R_{-\th}(Q^\op,\bv).
$$
\end{rem}

\section{Framings}

Our exposition in this section will be close to the one given by Nakajima in
\cite{Na5}.

\subsection{}\label{cb_frame}
The set $\R_\th(Q,\bv)$ is often empty in various interesting
cases of
quivers $Q$ and dimension vectors $\bv$.
Introducing a framing is a way to remedy the situation.

To explain this, fix a quiver $Q$ with vertex set $I$. We introduce another quiver
$\fq$, called the {\em framing of}, $Q$ as follows.
The set of vertices of $\fq$ is defined to be
$I\sqcup I'$, where $I'$ is another copy of the set $I$,
 equipped with the bijection $I\iso I',\ i\mto i'.$
The set of edges of $\fq$ is, by definition,
a disjoint union of the set of edges of $Q$ and
a set of 
additional edges $\bj_i: i\to i',$ from the vertex $i$ to the
corresponding vertex $i'$, one for each vertex $i\in I$.

Thus, giving a representation of $\fq$ amounts to giving
a representation $\bx$, of the original quiver $Q$, in a vector space
$V=\oplus\ii V_i$
 together with
a collection of linear maps
$V_i\to W_i, \ i\in I,$ where  $W=\oplus\ii W_i$ is
an additional collection of finite dimensional vector spaces, where
$W_i$
is `placed'
at the vertex $i'\in I'$. We let $\bw:=\dim_I W\in\Z^I_{\geq0}$ denote the
corresponding dimension vector,
and write  $\bj: V\to W$ to denote a collection
of linear maps $\bj_i:\ V_i\to W_i,\ i\in I,$ as above.

With this notation, a representation of $\fq$ 
is a pair $(\bx,\bj)$, where $\bx$ is a representation of $Q$ in
$V=\oplus_i\ V_i$, 
 and $\bj: V\to W$ is arbitrary additional collection of linear maps.
Accordingly, dimension vectors for the quiver $\fq$ are elements
$\bv\times\bw\in \Z^I\times \Z^I=\Z^{I\sqcup I'}.$
We write $\Rep(\fq,\bv,\bw):=\Rep(\fq, \bv\times\bw)$ for the space of representations
 $(\bx,\bj)$, of $\fq$, of dimension $\dim_I V=\bv,\ \dim_I W=\bw$.

We define a $G_\bv$-action on  $\Rep(\fq,\bv,\bw)$ by
$g:\ (\bx,\bj)\mto (g\bx g\inv,\bj\ccirc g\inv),$ where we write
$\bj\ccirc g\inv$ for the collection of maps $ V_i\stackrel{(g_i)\inv}\too
V_i\stackrel{\bj_i}\too W_i.$

\begin{rem} The group $G_\bv=\prod\ii\ GL(V_i)$ 
may be viewed as a subgroup in $G_\bv\times G_\bw=
\prod\ii\ GL(W_i)\times \prod\ii\ GL(W_i)$.
The later group acts on $\Rep(\fq,\bv,\bw)$
according to the general rule of \S\ref{reminder}
applied in the case of the quiver $\fq$.
The $G_\bv$-action defined above is nothing but
the restriction of the  $G_\bv\times G_\bw$-action to
the subgroup $G_\bv$.

From now on, we will view $\Rep(\fq,\bv,\bw)$
as a $G_\bv$-variety, and will ignore the action of the other factor,
the group $G_\bw$.
\erem

There is a slightly different but equivalent point of view
on framings, discovered by Crawley-Boevey \cite{CB1}, p.261.
Given a quiver $Q$, with vertex set $I$, and a dimension vector
$\bw=(w_i)\ii\in\Z^I_{\geq0}$,  Crawley-Boevey  considers a quiver
$Q^\bw$ with the vertex set equal to $I\sqcup\{\infty\}$,
where $\infty$ is a new additional vertex.
The set of edges of the quiver $Q^\bw$ is obtained from the set of edges of
$Q$ by adjoing $w_i$ additional edges $i\to \infty$, for each vertex
$i\in I$. 

Next, associated with any dimension vector $\bv=(v_i)\ii\in\Z^I_{\geq0}$,
introduce a dimension vector $\wh\bv\in \Z^{I\sqcup\{\infty\}}_{\geq0}$
such that $\wh v_i:=v_i$ for any $i\in I$, and $v_\infty:=1.$
We have a natural group imbedding $G_\bv\into G_{\wh\bv}$
that sends an element $g=(g_i)\in\prod\ii GL(v_i)$ to the element
$\wh g=(\wh g_i)_{i\in I\sqcup\{\infty\}}\in \prod_{i\in I\sqcup\{\infty\}} GL(\wh v_i)$,
where
$\wh g_i:=g_i$ for any $i\in I$ and $g_\infty:=\Id$.
Note that this imbedding induces an isomorphism
$G_\bv\iso G_{\wh\bv}/\gm$.
Thus, we may (and will) view 
$\Rep(Q^\bw,\wh\bv)$ as a $G_\bv$-variety via that isomorphism.

Let
$\wh V=\oplus_{\{I\sqcup\{\infty\}}\ V_i$, resp.  $W=\oplus\ii W_i$, be  a vector space
such that $\dim_I \wh V=\wh\bv,$ resp.  $\dim_I W=\bw$.
We identify $V_\infty$ with $\C$, a 1-dimensional vector space
with a fixed base vector.

For each  $i\in I$, we choose a basis of the
vector space $W_i$. Then, given
any collection of $w_i$ linear maps $V_i\to V_\infty$
(corresponding to the $w_i$ edges $i\to\infty$
of the quiver $Q^\bw$),
  one can use the basis in $W_i$ to assemble these maps into a single linear map
$\bj_i: V_i\to W_i$. This way, we see that
any representation of the quiver $Q^\bw$ in the vector space $\wh V$ gives
rise to a representation of  the quiver $\fq$,
that is, to a point in $\Rep(\fq,V,W)$.
The resulting map $\Rep(Q^\bw,\wh\bv)\to \Rep(\fq,\bv,\bw)$
is a $G_\bv$-equivariant  vector space isomorphism that depends
on the choice of basis  of the vector space $W$.
However, the morphisms corresponding to different choices of basis
are obtained from each other by composing with an invertible
linear map $g: \Rep(\fq,\bv,\bw)\to \Rep(\fq,\bv,\bw)$
that comes from the action on  $\Rep(\fq,\bv,\bw)$ of
an element $g$ of
the group $G_\bw$.

\subsection{Stability for framed representations}
We may apply the general notion of stability in GIT,
cf.  Definition \ref{def_stab}, in the special case of the
$G_\bv$-action on the variety  $\Rep(\fq,\bv,\bw)$
and a character $\chi_\th:\ G_\bv\to\gm$.

The notion of (semi)stability for {\em framed}
representations
of the quiver $\fq$
in the sense of Definition  \ref{def_stab} may {\em not} agree with
the notion of (semi)stability  for  
representations
of $\fq$  in the sense of Definition \ref{stab}. This is because
the general  Definition  \ref{def_stab} refers to a choice
of group action. Considering a representation of $\fq$ as an
ordinary representation without framing refers implicitly
to the action of the group $G=G_\bv\times G_\bw$,
while considering the same representation as a framed representation
refers to the action of the group $G=G_\bv$.

Let $\th\in\BR^I$.
A convenient way to relate the $\th$-stability of framed representations to
King's results is to use the $G_\bv$-equivariant isomorphism
  $\Rep(Q^\bw,\wh\bv)\to \Rep(\fq,\bv,\bw)$ described at the end of the
previous subsection.
Recall that we have $G_\bv\cong G_{\wh\bv}/\gm$,
where the subgroup $\gm\sset G_{\wh\bv},$ of scalar matrices, acts trivially
on $\Rep(Q^\bw,\wh\bv)$. Further,
define a vector
$\wh\th\in\BR^{I\sqcup\{\infty\}}$ by
$\wh\th_i:=\th_i$ for any $i\in I$ and
$\wh\th_\infty:=-\sum_{i\in I}\th_i\cdot v_i$.
In this way, all the results
 of  \S\ref{king_sec} concerning $\wh\th$-stability for the
$G_{\wh\bv}$-action on  $\Rep(Q^\bw,\wh\bv)$
may be transferred into corresponding results concerning $\th$-stability for the
$G_\bv$-action on $\Rep(\fq,\bv,\bw)$.

\begin{rem}
Note  that the  $G_\bv$-action on $\Rep(\fq,\bv,\bw)$ does {\em not}
factor through the quotient $G_\bv/\cm$. 

Observe also that our definition of the vector
$\wh\th$ insures that one has $\wh\bv\cdot\wh\th=0.$
\erem

Below, we restrict ourselves to the special case
of the vector
\beq{th}\th^+:=(1,1,\ldots,1)\in \Z^I_{>0}.
\eeq

We write
`semistable' for `$\th^+$-semistable', and let $\Rep^{ss}(\fq,\bv,\bw)$
denote the set of semistable representations of $\fq$ of dimension $(\bv,\bw)$.
Further let $\R(\bv,\bw):=\Rep^{ss}(\fq,\bv,\bw)/\!/_{\chi_{\th^+}}G_\bv$ be
the corresponding
GIT quotient.

We have the following result.

\begin{lem}\label{stabfq} \vi A representation
$(\bx,\bj)\in\Rep(\fq,\bv,\bw),$
in vector spaces $(V,W)$,
is semistable (with respect to the  $G_\bv$-action on
$\Rep(\fq,\bv,\bw)$)
if and only if  there is no nontrivial subrepresentation
$V'\sset V$, of the quiver $Q$, contained in $\Ker \bj$.

\vii The group $G_\bv$ acts freely on the set  $\Rep^{ss}(\fq,\bv,\bw)$, 
moreover,
any semistable  representation is automatically stable.

\viii $\R(\bv,\bw)$ is a smooth quasi-projective variety and the canonical map
$\Rep^{ss}(\fq,\bv,\bw)/G_\bv$
$\to\R(\bv,\bw)$ is a bijection of sets.
\end{lem}
\begin{proof} Part (i) follows by directly
applying  Theorem
\ref{kingthm} to the $G_{\wh\bv}$-action on  $\Rep(Q^\bw,\wh\bv)$. To prove (ii),  let  $g\neq \Id$
be an element of  the isotropy group
of a representation $V\in \Rep^{ss}(\fq,\bv,\bw)$.
Then, $V':=\Ker(g-\Id)$ is a subrepresentation of $V$ that violates
the condition of part (i).
Part (ii) follows from this. Part (iii) follows from
(ii) by Corollary \ref{sss}.
\end{proof}

\begin{prop}[King]\label{king3} \vi Assume that $Q$ has no edge loops and the set of
$\th$-stable $(\bv,\bw)$-dimensional framed representations of $Q$  is
nonempty. Then, we have
\beq{dimR} \dim \R_\th(\bv,\bw)=\bv\cd\bw+A_Q\bv,\bv-\bv\cdot\bv,
\eeq

\vii If $Q$ has no oriented cycles
then the scheme $\R_\th(\bv,\bw)$ is a (smooth) projective
variety.
\end{prop}
\proof[Sketch of proof of formula \eqref{dimR}]
Observe first that we have
$$\dim\Rep(\fq,\bv,\bw)=\bw\cdot\bv+A_Q\bv\cdot\bv.$$
Furthermore, one shows that, for $\th$ as in the
statement of the proposition.
the set $\Rep^s_\th(\fq,\bv,\bw)$ is Zariski open in
$\Rep(\fq,\bv,\bw)$. The $G_\bv$-action on
$\Rep^s_\th(\fq,\bv,\bw)$ being free, we compute
\begin{multline*}
\dim  \R_\th(\bv,\bw)=\dim\big(\Rep^s(\fq,\bv,\bw)/G_\bv\big)\\
=
\dim\Rep^s(\fq,\bv,\bw)-\dim G_\bv=
\dim\Rep(\fq,\bv,\bw)-\dim G_\bv\\=
\bw\cdot\bv+A_Q\bv\cdot\bv-\bv\cdot\bv.
\tag*{$\Box$}
\end{multline*}

\begin{examp}[Jordan quiver] Let $Q$ be a quiver with a single vertex
and a single edge-loop at this vertex.
For any positive integers $n,m\in\Z^I=\Z,$ we have $\Rep(Q,n)=\End \C^n$.
Further, we have
$$ \fq:\qquad
\xymatrix{
\bullet\ar@(ul,dl)\ar[r]^<>(0.5){\bj}&\bullet.
}
$$
Hence, we get $\Rep(\fq,n,m)=\End \C^n\times \Hom(\C^n,\C^m)$.

First, let $m=0$, so we are considering representations of
$Q$, not of $\fq$.
It is clear that, for $\th=\th^+=1$, any $n$-dimensional representation
of $Q$  is $\th$-semistable.
There are no stable representations unless $n\leq 1.$.

Let $\BS_n$ denote the Symmetric group and let
$\C^n/\BS_n$ be
the set of unordered
$n$-tuples of complex numbers viewed as an affine variety.
The
map sending an $n\times n$-matrix to the (unordered)
$n$-tuple of its eigenvalues yields an isomorphism
$\R(n)=\Rep(Q,n)/\!/GL_n\iso\C^n/\BS_n$.

Now, take $m=1$, so we get $\Rep(\fq,n,m)=\End \C^n\times
\Hom(\C^n,\C)$. A pair $(\bx,\bj)\in \End \C^n\times(\C^n)^*$
is semistable if and only if  the linear function $\bj: \C^n\to\C$ is a
{\em cyclic vector} for $\bx^*: (\C^n)^*\to (\C^n)^*$, the dual
operator. 

It is known that the $GL_n$-action on the
set $\Rep^s(\fq, n,1)$, of such pairs
$(\bx,\bj)$, is free. Moreover,
sending $(\bx,\bj)$ to  the unordered
$n$-tuple of the eigenvalues of $\bx$ yields
a bijection between the set of $GL_n$-orbits
in $\Rep^s(\fq, n,1)$ and $\C^n/\BS_n$.
Thus, in this case, we have isomorphisms
$\R(\fq, n,1)\cong\R^s(n)\cong\C^n/\BS_n$.
\end{examp}

\begin{examp}[Type $\textbf{A}$ Dynkin quiver]\label{dyn}
$$ Q:\
\xymatrix{
\ \stackrel{1}\bullet\ &\ \stackrel{2}\bullet\
\ar@<1ex>[l]&\ldots\ar@<1ex>[l]&\ 
\stackrel{n-2}\bullet\ \ar@<1ex>[l]&
\ \stackrel{n-1}\bullet\ \ar@<1ex>[l]&\ \stackrel{n}\bullet\ \ar@<1ex>[l]
}
$$

In this case, we have $I=\{1,2,\ldots,n\}$ and $\Rep(Q,\bv)/\!/GL_\bv=pt,$
since $Q$ has no oriented cycles. 

We let $\bv=(v_1,v_2,\ldots,v_n)$ and $\bw=(r,0,0,\ldots,0)$,
where $r>v_1>v_2>\ldots>v_n>0,$ is a strictly decreasing sequence
of positive integers. An element of $\Rep(\fq,\bv,\bw)$
has the form
 $(\bx,\bj)$,
where $\bx=(x_{i-1,i}: \C^{v_i}\to\C^{v_{i-1}})_{i=2,\ldots,n}$,
and  the only nontrivial component
of $\bj$ is a linear map $j:=\bj_1: \C^{v_1}\to\C^r$.

Observe that the collection of vector spaces
$$F_i:=\Image(j\ccirc \bx_{21}\ccirc \ldots\ccirc  \bx_{i-1,i})\sset \C^r,
\qquad
i=1,\ldots,n,
$$
form an $n$-step partial flag, $F_1\sset F_2\sset\ldots\sset F_n=\C^r$,
in $\C^r$.
Now, the stability condition amounts, in this case, to the {\em injectivity} of
 each of the maps
$j,\ \bx_{12},\ldots,\ \bx_{n-1,n}$.
It follows that  we have
$$\dim F_i=\dim\Image(j\ccirc \bx_{21}\ccirc \ldots\ccirc  \bx_{i-1,i})=
\dim \C^{v_i}=v_i.
$$

Let $\FF(n, W)$ be the
variety formed by $n$-step partial flags $F=(F_1\sset
F_2\sset\ldots\sset F_n=W),$ such that
$\dim F_i=v_i,\, \forall i\in I.$
In this way, for the corresponding
moduli space, one obtains an isomorphism 
$\R(\bv,\bw)\cong\FF(n, W).$ 
In particular, $\R(\bv,\bw)$ is a smooth projective variety,
in accordance with Proposition \ref{king3}(ii).
\end{examp}

\section{Hamiltonian reduction for representations of quivers}\label{cot_sec}
\subsection{Symplectic geometry} To motivate later constructions, we first
remind a few basic definitions.

Let $X$ be a smooth manifold,
 write $T^*X\to X$ for the  the cotangent bundle on $X$.
The
total
space $T^*X$, of the cotangent bundle, comes equipped
with a canonical symplectic structure, i.e. there is
a canonically defined nondegenerated closed 2-form
$\om$ on $T^*X$. 

In the case where $X$ is a vector space,
the only case we will use below, we have $T^*X=X\times X^*$,
where $X^*$ denotes the vector space dual to $X$.
The canonical symplectic structure on $X\times X^*$
is given, in this special case, by a constant 2-form defined by the formula
\beq{omega}
\om(x\times x^*,\ y\times y^*):=
\langle y^*,x\rangle-\langle x^*,y\rangle,\quad\forall x,y\in X,\ x^*,y^*\in X^*,
\eeq
where $\langle-,-\rangle$ stands for the canonical pairing between
a vector space and the dual vector space.

Now, let a Lie group $G$ act on  an arbitrary smooth manifold $X$.
Let $\g$ be the  Lie algebra of $G$.
Given $u\in \g$, write $\arr{u}$ for the
 vector field on $X$ corresponding to the
`infinitesimal $u$-action' on $X$, and let
$\arr{u}_x$ be the value of that vector field
at a point $x\in X$.

Associated with the $G$-action on $X$,  there is a
natural $G$-action on $T^*X$ and a
canonical {\em moment map}
\beq{moment}
\mu:\ T^*X \to\g^*,\quad \al_x\mto \mu(\al),
\quad\text{defined by}\en \g^*\ni\mu(\al_x): u\mto \langle\al,\arr{u}_x\rangle,
\eeq
where $\al_x\in T^*_xX$ stands for a covector at  a point $x\in X$.

The following properties of the map \eqref{moment} are straightforward  consequences of the
definitions.

\begin{prop}\label{moment_basic}
\vi If the group $G$ is {\em connected} then the moment map 
is $G$-equivariant, i.e. it intertwines 
the $G$-action on $T^*X$ and the coadjoint $G$-action on $\g^*$.
\vs

\vii Writing $T^*_YX$ for the conormal bundle of a submanifold $Y\sset
X$,
one has
\beq{conormal}
\mu\inv(0)=\bigcup_{Y\in X/G}\ T^*_Y(X).
\eeq
\end{prop}

\noindent
Here, $X/G$ stands for the set  of
$G$-orbits on $X$. 

From the last formula one easily
derives the following result.

\begin{cor} Assume that the Lie group $G$ acts freely on $X$, and that
the orbit space $X/G$ is a well defined smooth manifold.
Then, \vs

\npb{The $G$-action on $T^*X$ is free, and the moment
map \eqref{moment} is a submersion.} \vs

\npb{For any  coadjoint orbit $\O\sset\g^*$, the orbit space
$\mu\inv(\O)/G$ has a natural structure of smooth symplectic manifold.} \vs

\npb{For $\O=\{0\},$ there is, in addition, a canonical symplectomorphism}
\beq{zer}T^*(X/G)\cong\mu\inv(0)/G.
\eeq
\end{cor}

Formula \eqref{zer} explains the importance of the zero fiber of the moment map.
Later on, we will consider quotients of $\mu\inv(0)$ by
the group action in situations where the group action on $X$ is no longer
free, so the naive orbit set $X/G$ can not be equipped with a reasonable
structure of a manifold. In those cases, various quotiets
of $\mu\inv(0)$ by $G$ involving stability conditions serve
as substitutes for the contangent bundle on a nonexisting 
space $X/G$.

The above discussion was in the framework of differential geometry,
where  `manifold' means a $C^\infty$-manifold. 
There are similar constructions and results in the algebraic geometric framework
where $G$ now stands for an  affine algebraic group and
$X$ stands for a $G$-variety.

 For any affine algebraic group $G$, the differential of a rational group homomorphism
$G\to\gm$ gives a linear function $\g\to\C$, i.e.
a point  $\la\in\g^*$. The points of $\g^*$ arising in this way
are automatically fixed by the coadjoint action of $G$ on $\g^*$.
If the group $G$ is connected, then the 
 corresponding fiber $\mu\inv(\la)$ is necessarily
a $G$-stable subvariety,
by Proposition \ref{moment_basic}(i). 
The varieties of that form play the
role of `twisted cotangent bundles' on $X/G$, cf. \cite{CG}, Proposition
1.4.14 and discussion
after it.
These varieties share many features
of the zero fiber of the moment map.

The following elementary result will be quite useful in applications
to quiver varieties.

\begin{lem}\label{smooth_fiber} {\em Let $\la\in\g^*$ be a fixed point of the
 coadjoint action of a connected group $G$, and let $G$ act on 
a manifold $X$ with an associated moment map $\mu$ as in \eqref{moment}.
Then, the following holds:}

A geometric point $\al\in\mu\inv(\la)$ is a smooth point  of the scheme 
theoretic fiber $\mu\inv(\la)$
 if and only if  $\al$
has finite isotropy in $G$. In such a case, the symplectic form on $T^*X$
induces
a  nondegenerate  bilinear form on the vector space
 $T_\al(T^*X)/\Lie G^\al$.
\end{lem}

\begin{proof} Put $M:=T^*X$, for short,
let $\al\in M$, and write  $G^\al\sset G$
for the isotropy group of the point $\al$.
Further, let $d_\al\mu:\ T_\al M\to\g^*$ stand for the differential of
the moment map
$\mu$ at the point $\al$.

Now let  $u\in\g$
and write $\arr{u}_\al\in T_\al M$ for the tangent vector
corresponding to the infinitesimal $u$-action on $M$.
Also, one may view $u\in\g$ as a linear function on $\g^*$.
The crucial observation, that follows directly from the definition
of the moment map, cf. \eqref{moment},  is that one has 
\beq{seq}\langle d_\al\mu(v),\ u\rangle=d_\al(u\ccirc\mu)(v)=
\om(\arr{u}_\al, v),\quad\forall u\in\g,\ v\in T^*_\al M.
\eeq

Using this, we deduce
\begin{align*}
\text{$G^\al$ is finite}\quad&\Longleftrightarrow\quad
\Lie G^\al=0\\
&\Longleftrightarrow\quad
\text{$\arr{u}_\al\neq0$ for any $u\neq
0$}\\
&\Longleftrightarrow\quad
\text{There is no $u\in\g,\ u\neq0,$ such that $\langle d_\al\mu(v),\ u\rangle=\om(\arr{u}_\al, v)=0$}\\
&\Longleftrightarrow\quad
\text{$d_\al\mu$ is surjective}\\
&\Longleftrightarrow\quad
\text{$\al$ is a smooth point.}
\end{align*}

This proves the first statement of the lemma. 
The  second
statement easily follows from \eqref{seq}
by similar arguments. We leave details to the
reader.
\end{proof}

\subsection{} Fix a finite set $I$ and a dimension vector
$\bv=(v_i)_{i\in I}\in\Z^I$.

From now on, we specialize to the case where the algebraic group $G$ is a product of general
linear groups, ie. is a group of the form $G_\bv=\prod_{i\in I}\ GL(v_i).$ Thus, we have 
$\g_\bv:=\Lie G_\bv=\oplus_{i\in I}\ \gl(v_i).$
The center of each summand $\gl(v_i)$ is a 1-dimensional
Lie algebra of scalar matrices.
Therefore, the center of $\g_\bv$ may be identified with the vector space
$\C^I$. 

Observe further that any Lie algebra homomorphism
$\g_\bv\to\C$ has the form
$\bx=(x_i)\ii\mto \sum\ii \la_i\cdot\Tr x_i$.
We deduce that the fixed point set of the coadjoint $G_\bv$-action
on $\g_\bv^*$ is a vector space  $\C^I\sset \g_\bv^*$.
Explicitly, an element
$\la=(\la_i)\ii\in \C^I$ corresponds to the point
in $\g^*_\bv$ given by the linear function
$\bx\mto \la\cdot\bx=  \sum\ii \la_i\cdot\Tr x_i$, on $\g_\bv$.

\subsection{The double $\dq$} Given a quiver $Q$, let $\dq= Q\sqcup Q^\op$
be the {\em double} of $Q$, the quiver that has
the same vertex set as $Q$ and whose set of edges is 
a disjoint union of the sets of edges of $Q$ and of $Q^\op$,
an opposite quiver.
Thus, for any edge $x\in Q$, there is a reverse edge
$x^*\in Q^\op\sset\dq$.

\begin{defn}\label{preproj} For any $\la=(\la_i)\ii\in\C^I$,
let $\Pi_\la=\Pi_\la(\dq)$ be a quotient of the
path algebra  $\C\dq$, of the double quiver $\dq$, by the two-sided ideal
generated by the following element
$$ \sum_{x\in Q}\ (xx^*-x^*x) -
\sum_{i\in I}\ \la_i\cd1_i\ \in\,\C \dq.
$$
Thus, $\Pi_\la(\dq)$ is an associative algebra called 
{\em preprojective algebra} of $Q$ with parameter $\la$.
\end{defn}

The defining relation for the  preprojective algebra
may be rewritten more explicitly
as a collection of relations,
one for each vertex $i\in I$, as follows:
$$
\sum_{\{x\in Q:\ \text{head}(x)=i\}} xx^*\en-\en\sum_{\{x\in Q:\
\text{tail}(x)=i\}}x^*x
\en=\en\la_i\cd1_i,
\qquad i\in I.
$$

Clearly, one has
$\Rep(\dq,\bv)\cong\Rep(Q,\bv)\times \Rep(Q^\op,\bv)$.
We will write a point of $\Rep(\dq,\bv)$
as a pair $(\bx,\by)\in \Rep(Q,\bv)\times \Rep(Q^\op,\bv)$.

Recall that, for any pair, $E,F$, of finite dimensional vector spaces,
there is a canonical perfect pairing
$$\Hom(E,F)\times\Hom(F,E)\to\C,\quad
f\times f' \mto \Tr(f\ccirc f')=\Tr(f'\ccirc f).
$$
Using this pairing, one obtains  canonical
isomorphisms of vector spaces
\beq{pairing}\Rep(Q^\op,\bv)\cong\Rep(Q,\bv)^*,
\quad\text{resp.}
\quad
\g_\bv\cong\g_\bv^*.
\eeq

We deduce the following isomorphisms
\beq{dq}
\Rep(\dq,\bv)
\cong\Rep(Q,\bv)\times 
\Rep(Q,\bv)^*\cong T^*\big(\Rep(Q,\bv)\big).
\eeq

The  natural $G_\bv$-action on $\Rep(\dq,\bv)$ corresponds,
via the  isomorphisms above, to the 
 $G_\bv$-action on the cotangent bundle induced by the
 $G_\bv$-action on $\Rep(Q,\bv)$.
Associated with the latter action, there is a moment map $\mu$.
It is given by the following explicit formula,
a special case of formula \eqref{moment}:
$$
\mu:\ 
\Rep(\dq,\bv)= T^*\big(\Rep(Q,\bv)\big)\too
\g_\bv^*=\g_\bv,\quad
(\bx,\by)\mto\sum\ (x\ccirc y-y\ccirc x)\ \in\ \g_\bv.
$$

We explain the above formula in the simplest case of
the Jordan quiver.
 
\begin{examp}\label{comm_examp1} Let $Q$ be a quiver with one vertex and one edge-loop.
Then, $\dq$ is a quiver with a single vertex and two edge-loops at that vertex.
Thus, given a positive
 integer $\bv\in\Z^I=\Z$, we have $\Rep(\dq,\bv)=\gl_\bv\times\gl_\bv$.
The action of the group $G_\bv$ on the space
$\Rep(\dq,\bv)$ becomes, in this case, the
$\Ad GL_\bv$-diagonal action on pairs of $(\bv\times\bv)$-matrices.

Further, the isomorphism  $\g_\bv\iso\g_\bv^*,$
resp. $\Rep(Q^\op,\bv)\iso\Rep(Q,\bv)^*,$  sends
a matrix $\bx\in \g_\bv$ to a linear function $\by\mto\Tr(\bx\cdot\by).$
Hence, in the notation of \S\ref{cot_sec},
for any $u\in \gl_\bv$, we have $\arr{u}=\ad u$.

Now, according to  definitions, see formula \eqref{moment},  the moment map 
sends a point $(\bx,\by)\in T^*(\gl_\bv)=\gl_\bv\times\gl_\bv$
to a linear function 
$$\mu(\bx,\by):\
\gl_\bv\to\C,\quad u\mto \langle \by,\  \arr{u}_\bx\rangle=
\langle \by,\  \ad u(\bx)\rangle=\Tr\big(\by\cdot [u,\bx]\big)=
\Tr\big([\bx,\by]\cdot u\big).
$$

We see that the linear function 
$\mu(\bx,\by)\in\gl_\bv^*$ corresponds, under
the isomorphism $\gl_\bv^*\iso\gl_\bv$, to the
matrix $[\bx,\by]$. 
We conclude
that the moment map for the $\Ad GL_\bv$-diagonal action on 
$T^*(\gl_\bv)=\gl_\bv\times\gl_\bv$ has the following final form
$$
\mu:\
\gl_\bv\times\gl_\bv\too\gl_\bv,\quad
\bx\times\by\mto [\bx,\by].
$$
This is nothing but the general formula \eqref{dq} in 
a special case of the Jordan quiver $Q$.\qq
\end{examp}

In general, it is clear from Definition \ref{preproj} that,  inside
 $\Rep(\bv,\dq)$, one has an
equality:
\beq{pirep}
\Rep(\Pi_\la, \bv)=\mu\inv(\la):=\{(\bx,\by)\in \Rep(\dq,\bv)\mid
[\bx,\by]=\la\},\qquad\la\in\C^I.
\eeq
This is, in fact, an isomorphism of schemes. 

\begin{rem}
Observe  that, for any 
$\Pi_\la$-representation $V$ of dimension $\bv$, in view
of the defining relation for the preprojective algebra, one 
must have
$$
\la\cdot\bv=\sum_{i\in I}\ \la_i\cdot\Tr_V(1_i)\\
=\Tr_V\left(\sum_{i\in I}\ \la_i1_i\right)=
\Tr_V\left(\sum_{x\in Q}\ (xx^*-x^*x)\right)
=0,
$$
where in the last equation we have used that the trace of any
commutator
vanishes. We deduce that the algebra $\Pi_\la$ has no
$\bv$-dimensional representations unless $\la\cdot\bv=0.$

This is consistent with \eqref{pirep}. Indeed,
the group $\gm\sset G_\bv$ acts trivially on 
$\Rep(\dq,\bv)$, hence
the image of the moment map $\mu$ is contained
in the hyperplane $(\Lie\gm)^\perp\sset\g^*_\bv$.
Therefore, the fiber $\mu\inv(\la)$ over a point
$\la\in \C^I\sset\g^*_\bv$ is empty  unless
we have $\la\cdot\bv=0$.
\erem

\begin{rem} It is important to
emphasize that, up to a relabelling $\la\mto\la'$ of parameters, one
has:

{\textbf{The quiver $\dq$, hence also the scheme $\mu\inv(\la)$ and the algebra $\Pi_\la(\dq)$,
depend only on the underlying graph of $Q$,
and do not depend on the orientation of the quiver $Q$.}}
\end{rem}

\subsection{}\label{structure}
The cotangent bundle projection
$p:\ T^*\big(\Rep(Q,\bv)\big)\to\Rep(Q,\bv)$
may be clearly identified with the natural projection
to $\Rep(\dq,\bv)\to\Rep(Q,\bv),\ (\bx,\by)\mto\bx$,
cf. \eqref{dq}. 
Restricting
the latter projection  to a fiber of the moment map
one obtains a map $p_\la:\ \Rep(\Pi_\la,\bv)=\mu\inv(\la)\to
\Rep(Q,\bv)$.

Observe further that
 the  composite $\C Q \into \C \dq \onto \Pi_\la$
 yields an algebra imbedding $\C Q \into \Pi_\la$.
In terms of the latter imbedding, the map
$p_\la$ amounts to restricting representations
of the algebra $\Pi_\la$ to the subalgebra $\C Q$.
Thus, we obtain, cf. \cite[Lemma 4.2]{CBH},

\begin{prop}\label{extend} For any  $\bx\in \Rep(Q,\bv),$
the set $p_\la\inv(X)$ is canonically identified
with the set of extensions 
of $\bx$ to a $\Pi_\la$-module $(\bx,\by)\in \Rep(\Pi_\la,\bv)$.\qed
\end{prop}

In some important cases, one can say quite a bit about the structure of the
variety \eqref{pirep}. To explain this, we need to introduce some
notation.

Let $R_Q\sset \Z^I$ be the set of roots for $Q,$ as defined eg. in  \cite{CB1}, p. 262.
Given $\la\in\C^I,$ we put $R^+_\la:=\{\al\in R_Q\mid \al\geq 0\ \&\
\la\cdot\al=0\}$ where, in general, we 
write $\bv\geq\bv'$ whenever
$\bv-\bv'\in\Z^I_{\geq 0}$.
Further, for any $\bv\in\Z^I$, we define
$${\mathsf{p}}(\bv):=1+A_Q\bv\cdot\bv-\bv\cdot\bv.$$

Recall the hyperplane $(\Lie\gm)^\perp\sset\g^*_\bv$
that corresponds to the diagonal imbedding $\gm\sset G_\bv$.
One has the the following result.

\begin{thm}\label{gg} Fix $\la\in\C^I$. 
Let $\bv$ be a dimension vector such that $\la\cdot\bv=0$ and,
for any decomposition $\bv=\al_1+\ldots+\al_r,\
\al_j\in R^+_\la$, the following
inequality holds
\beq{inequality}
{\mathsf{p}}(\bv)\geq {\mathsf{p}}(\al_1)+\ldots+{\mathsf{p}}(\al_r).
\eeq

Then, we have\vs

\vi The moment map $\mu:\Rep(\dq,\bv)\to(\Lie\gm)^\perp$ is flat
and the scheme $\Rep(\Pi_\la,\bv)$, in \eqref{pirep}, is a complete intersection in
$\Rep(\dq,\bv)$.\vs

\vii The irrducible components of  $\Rep(\Pi_\la,\bv)$ are in one-to-one
correspondence with decompositions $\bv=\al_1+\ldots+\al_r,\
\al_j\in R^+_\la,$ such that the corresponding inequality \eqref{inequality}
is an equality. Each irreducible component has dimension
$1+2A_Q\bv\cdot\bv-\bv\cdot\bv$.\vs

\viii If the inequality  \eqref{inequality} is strict for
any $\bv=\al_1+\ldots+\al_r,\
\al_j\in R^+_\la,$ with $r>1$, then the scheme  $\Rep(\Pi_\la,\bv)$
is reduced and irreducible, moreover, the general
point in this scheme is a simple representation of the
algebra $\Pi_\la$.
\end{thm}

Here, parts (i) and (iii) are due to Crawley-Boevey,
\cite{CB1}, Theorems 1.1 and 1.2. Part (ii) is
\cite{GG}, Theorem 3.1.

\subsection{Hamiltonian reduction}
For any $\la\in\C^I$ such that $\la\cdot\bv=0$, the fiber $\mu\inv(\la)$
is a nonempty closed $G_\bv$-stable subscheme of $\Rep(\dq,\bv)$,
not necessarily reduced, in general. Thus, given $\th\in\BR^I$ 
such that $\th\cdot\bv=0$,
one may consider the following GIT quotient
\beq{M}
\M_{\la,\th}(\bv):=\mu\inv(\la)/\!/_{\chi_\th}G_\bv=
\Rep^{ss}(\Pi_\la,\bv)/\text{$S$-equivalence},
\qquad\forall\, \la\cdot\bv=\th\cdot\bv=0.
\eeq

\begin{rem}\label{hyperkahler} One may identify $\C^I\times\BR^I=\BR^3\o \BR^I$ and view
a pair $(\la,\th)\in \C^I\times\BR^I$ as a point in
$\BR^3\o \BR^I$. Further, given
$\bv=(v_i)\ii$, view $\C^{v_i}$ as a hermitian vector space
with respect to the standard euclidean (hermitian) inner product.
These inner products induce
hermitian inner products on the spaces
$\Hom(\C^{v_i},\C^{v_j})$.
The resulting 
 hermitian inner product on
$\Rep(\dq,\bv)$
combined with the ($\C$-bilinear) symplectic 2-form, see \eqref{omega},
give  $\Rep(\dq,\bv)$ the structure of a 
hyper-K\"ahler vector space.

One can show, cf. \cite{Kro} for a special case,
that the Hamiltonian
reduction $\mu\inv(\la)/\!/_{\chi_\th}G_\bv$ may be
identified with a {\em hyper-K\"ahler reduction} of
 $\Rep(\dq,\bv)$
with respect to the maximal compact subgroup
of the complex algebraic group $G_\bv$ formed by the
elements which  preserve
the metric.
\end{rem}

To proceed further, we need to introduce  the {\em Cartan matrix} of the underlying graph of $Q$
defined as follows $C_Q:=2\Id-A_{\dq}.$ This is a symmetric Cartan
matrix in the sense of Kac, \cite{Ka}, provided $Q$ has no edge loops.

\begin{cor}\label{U/G} \vi Any simple $\Pi_\la$-module of dimension
$\bv$ corresponds to a point in $\mu\inv(\la)^{\oper{reg}}$, the smooth locus of
the  scheme \eqref{pirep}\vs

\vii The group $G_\bv/\gm$  acts freely on $\mu\inv(\la)^{\oper{reg}}$.\vs

\viii Let $T_{G_\bv\al}(\mu\inv(\la))$ be the normal
space, at $\al \in \mu\inv(\la)^{\oper{reg}}$,
 to the orbit $G_\bv\al\sset \mu\inv(\la)^{\oper{reg}}$.\vs

Then,
the vector space $T_{G_\bv\al}(\mu\inv(\la))$
has a canonical symplectic structure and,  we have
$$\dis\dim T_{G_\bv\al}(\mu\inv(\la))=2-C_Q\bv\cdot\bv.$$
\end{cor}
\begin{proof} Part (i) follows,  thanks to Schur's
lemma, from Lemma \ref{connected} and
Lemma \ref{smooth_fiber}. The last lemma also yields part
(ii).

To prove (iii), 
put $U:=\mu\inv(\la)^{\oper{reg}}$, let 
$G:=G_\bv/\gm$ and $\g:=\Lie G$. Thus, we have $\dim\g=\dim G_\bv-1$.

For any $\al\in U$, the
tangent space to $U/G$ at the point corresponding to
the image of $\al$ equals  $(T_\al U)/\g$,
where we identify the Lie algebra $\g$
with its image under the action map
$\g\to T_\al U,\ u\mto\arr{u}_\al$.
Furthermore, the (proof of)  Lemma \ref{smooth_fiber} implies
that this last map is injective.
Also, the  symplectic structure on $(T_\al U)/\g$
is provided by the last statement of Lemma \ref{smooth_fiber}.

Now, using the surjectivity of
the differential of the moment map
$d_\al\mu: T_\al\Rep(\dq,\bv)\to\g^*$ is surjective by   Lemma \ref{smooth_fiber},
 we compute 
\begin{align}
\dim U/G&=\dim\big((T_\al U)/\g\big)\nonumber\\
&=\dim\Ker(d_\al\mu)-\dim\g\label{Uf}\\
&=\big[\dim\Rep(\dq,\bv)-\dim\g^*\big]-\dim\g=\dim\Rep(\dq,\bv)-2\dim\g\nonumber\\
&=\dim\Rep(\dq,\bv)-2(\dim G_\bv -1).\nonumber
\end{align}

Finally,
from formula \eqref{dims} applied to the quiver $\dq$,
we find
\beq{cartan}
2\dim G_\bv-\dim\Rep(\dq,\bv)=2\bv\cdot\bv-A_{\dq}\bv\cdot\bv=
C_Q\bv\cdot\bv.
\eeq

The last formula of Corollary \ref{U/G}
now follows by combining \eqref{Uf} with \eqref{cartan}.
\end{proof}

Many of the results concerning stability of quiver representations
carry over in a straightforward way to $\Pi_\la$-modules.
 In particular, we have

\begin{thm}\label{Mthm} \vi For $\th=0$, the scheme
$\M_{\la,0}(\bv)=\Spec\C[\mu\inv(\la)]^{G_\bv}$ is a normal affine
variety, cf. \cite[Theorem 1.1]{CB3}; geometric
points of this 
scheme
correspond to semisimple $\Pi_\la$-modules.\vs

\vii Geometric points of the scheme
$\M_{\la,\th}(\bv)$ correspond to $S$-equivalence classes
of $\th$-semistable  $\Pi_\la$-modules. \vs

\viii The group $G_\bv$ acts freely on the
set $\mu\inv(\la)^s$, of $\th$-stable points;
isomorphism classes of 
$\th$-stable $\Pi_\la$-modules form a Zariski open subset
$\M^s_{\la,\th}(\bv)\sset \M_{\la,\th}(\bv),$ of dimension
$2-(\bv,C_Q\bv)$.\vs

\iv The canonical map $\pi:\
\M_{\la,\th}(\bv)\to \M_{\la,0}(\bv)$ is
a projective morphism that takes
a $\Pi_\la$-module $V$ to its semi-simplification.
\end{thm}
\begin{proof}[Sketch of Proof]
Part (i) of the theorem is a consequence of Corollary \ref{trivcor}.

To prove (iii), let $V\in \mu\inv(\la)^s$ be a  stable $\Pi_\la(\dq)$-module. A version
of Corollary \ref{stabrep}(ii) implies that
the isotropy group of $V$ is equal to $\gm$.
It follows that $V$ gives a smooth point of the fiber
$\mu\inv(\la)$, by Lemma \ref{smooth_fiber}.
Furthermore, Corollary \ref{U/G} applies and 
we find
$$\dim \M^s_{\la,\th}(\bv)= 2(\dim \Rep(Q,\bv)-\dim(G_\bv/\gm))=
2-(\bv,C_Q\bv).
$$
Other statements of the theorem are obtained
 by applying Theorem
\ref{kingthm} to the quiver ~$\dq$.
\end{proof}
\begin{cor} If the set $\M^s_{\la,\th}(\bv)$ is nonempty
then, we have $C_Q\bv\cdot\bv\leq
2$.\qed
\end{cor}

In the special case $\la=0$, using  isomorphism \eqref{conormal},
we deduce

\begin{prop} The variety $\M_{0,\th}(\bv)$ contains
$T^*\R^s_\th(\bv)$, the cotangent space to the moduli space 
$\R^s_\th(\bv)$, as an open (possibly empty) subset of the
smooth locus of $\M_{0,\th}(\bv)$.
\end{prop}

\begin{examp}[Dynkin quivers]\label{ADEb} Let $Q$ be an ADE quiver, and
fix a dimension vector $\bv$.

The number of $G_\bv$-orbits in $\Rep(Q,\bv)$ is finite by
the Gabriel theorem. Thus,
we see from \eqref{conormal} that $\mu\inv(0)$ is in this case
a finite union of conormal bundles, hence a Lagrangian subvariety
in $T^*\Rep(Q,\bv).$

We claim next that the zero representation $0\in \Rep(\dq,\bv)$
is contained in the closure of any $G_\bv$-orbit.
This is clear for the  $G_\bv$-orbit of  the point corresponding to an indecomposable
representation (if such a represntation of dimension $\bv$ exists),
since the corresponding orbit is Zariski dense by
the Gabriel theorem. From this, one deduces easily that our claim 
must hold for the orbit of a point corresponding to a direct sum of  indecomposable
representations as well.

Our claim implies that the conormal bundle on 
any $G_\bv$-orbit is a subset of $\Rep(\dq,\bv)$ which is
stable under the $\gm$-action on the vector space  $\Rep(\dq,\bv)$ by dilations
(this is not the action obtained by restricting
the natural $G_\bv$-action on $\Rep(\dq,\bv)$
to the subgroup $\gm\sset G_\bv$; the latter $\gm$-action
is trivial). It follows that the set $\mu\inv(0)$,
the union of conormal bundles,
is also $\gm$-stable under the dilation action.

We conclude that any homogeneous $G_\bv$-invariant polynomial
on $\mu\inv(0)$ of positive degree vanishes. Thus, in the Dynkin case,
we have
$
\M_0(\bv)=pt.
$
\end{examp}

\subsection{The McKay correspondence}\label{mksec} 
Associated to any finite group $\Ga$ and a finite dimensional
representation $\Ga\to GL(E)$, there is a quiver
$Q_\Ga$, called the {\em McKay quiver} for $\Ga$
(that depends on the representation $E$ as well).
The vertex set of this quiver is defined to be
the set $I$  of equivalence classes
of irreducible representations of $\Ga$. We write
$L_i$ for the  irreducible representation corresponding
to a vertex $i\in I$.
In particular, there is a distinguished vertex $o\in I$ corresponding
to the trivial representation.

Further, the adjacency matrix $A_{Q_\Ga}=\|a_{ij}\|$,
of the McKay quiver, is
defined by the formula
\beq{mck}a_{ij}:=\dim\Hom_\Ga(L_i, L_j\o E).
\eeq

The matrix $A_Q$ is symmetric if and only if
$E$ is a self-dual representation of $\Ga$.
In such a case, one can write
$Q_\Ga=\dq,$ for some quiver $Q$.
Note also that the quiver $Q_\Ga$ has no edge-loops
 if and only if
$E$ does not contain the trivial representation
of $\Ga$ as a direct summand.

Now, fix a 2-dimensional
symplectic vector space $(E,\om),$ and a
finite subgroup
$\Ga\sset Sp(E,\om)=SL_2(\C)$. The imbedding
$\Ga\into GL(E)$ gives a self-dual representation of $\Ga$.

It follows from the well known classification of
platonic solids that   conjugacy classes of finite subgroups of the group
$SL_2(\C)$ are in  one-to-one  correspondence with
Dynkin graphs of $A,D,E$ types.
McKay  observed that this correspondence may be
obtained by assigning to $\Ga\sset Sp(E,\om)$
its McKay quiver $Q_\Ga$ (so that  $Q_\Ga$ then becomes
the double of the corresponding  extended Dynkin graph
of type $\wt A,\ \wt D,\ \wt E,$
 equipped with any choice of orientation).

Associated with  the   extended Dynkin diagram,
there is a root system $R\sset \Z^I$.
The vector $\delta=(\delta_i)\ii\in\Z^I$,
where  $\delta_i:=\dim L_i,$ turns out to be equal
to the minimal imaginary root of that root system.

Let $\C \Ga$ be the group algebra of $\Ga$ and, for each $i\in I$,
choose a minimal idempotent $e_i\in\C \Ga$ such that
$\cg\cdot e_i\cong L_i$. Put $e=\sum\ii e_i,$ an idempotent in
$\cg$. The $\Ga$-action on $E$ induces one on
$\Sym E$, the symmetric algebra of $E$. We write $(\Sym E)\rtimes \Ga$ for the
corresponding cross-product algebra, resp.
$(\Sym E)^\Ga\sset \Sym E$ for the subalgebra of $\Ga$-invariants.
Note that the self-duality of $E$ implies that one has
algebra isomorphisms
 $(\Sym E)^\Ga=\C[E]^\Ga=\C[E/\Ga]$.

One way of stating the McKay correspondence is as follows,
cf. \cite{CBH}, Theorem 0.1.
\begin{thm} \vi There is an
algebra isomorphism $\Pi_0(Q_\Ga)\cong
e\big[(\Sym E)\rtimes \Ga\big]e$. In particular, 
the
 algebras $\Pi_0(Q_\Ga)$ and $(\Sym E)\rtimes \Ga$
are Morita equivalent.

\vii There is a canonical algebra isomorphism
$1_o\cdot\Pi_0(Q_\Ga)\cdot1_o \cong (\Sym E)^\Ga$.
\end{thm}

\begin{proof}[Outline of Proof]
Let $T_\C V$ denote the tensor algebra of the vector
space $E$.
An elementary argument based on formula \eqref{mck}
yields an algebra isomorphism
$\phi:\ T_{\C I}(\C E_{Q_\Ga})
\iso
e\big[(T_\C  E)\rtimes \Ga\big]e$, see \cite[\S2]{CBH}.
Recall that, for the path algebra
of any quiver $Q$, we have
$\C Q=T_{\C I}(\C E_Q),$ see \S\ref{reminder}.
Thus, we may identify the algebra on the left hand side of the
 isomorphism $\phi$
with the path algebra $\C Q_\Ga$. 

Next, one verifies that the
two-sided
ideal of  $\C Q_\Ga$ generated by the
element $\sum_{x\in Q_\Ga} (xx^*-x^*x)$,
cf. Definition \ref{preproj},
goes under the isomorphism $\phi$ to the
two-sided
ideal $J$ generated by the elements
$e_1\o e_2-e_2\o e_1\in T^2_\C E,\ e_1,e_2\in E$.
The isomorphism of part (i) of the theorem is now 
induced by the isomorphism $\phi$ using that one has
$$\C Q_\Ga/\big(\sum_{^{x\in Q_\Ga}} (xx^*-x^*x)\big)=\Pi_0(Q_\Ga),\quad
\text{and}\quad
[(T_\C E)\rtimes \Ga]/J=(\Sym E)\rtimes \Ga.
$$

To complete the proof, we observe that
the  isomorphism $\phi$ constructed above
sends the  idempotent $1_o\in \C I$ to
$\mathsf{p}:=\frac{1}{|\Ga|}\sum_{g\in\Ga} g\in \C \Ga$,
 the averaging idempotent.
Furthermore, it is easy to show that
the natural imbedding $(\Sym E)^\Ga\into
\Sym E$ induces 
an algebra isomorphism
$(\Sym E)^\Ga\iso$ $\mathsf{p}[(\Sym E)\rtimes \Ga]\mathsf{p}.$
Part (ii) of the theorem follows from these observations
using the isomorphism of part (i).
\end{proof}

The orbit space $\C^2/\Ga=\Spec\C[x,y]^\Ga$ is an irreducible normal
2-dimensional variety with an isolated singularity at the origin.
Such a variety is known to have a minimal resolution, unique
up to isomorphism.

The following result is a reformulation of a result of 
P. Kronheimer in the language of quiver moduli,
cf. \cite{CS}.

\begin{thm}\label{kronh} \vi There is a natural
isomorphism $\M_0(\delta)\cong\C^2/\Ga,$ of algebraic
varieties.

\vii Assume that $\th\in\Z^I$ does not belong to root
hyperplanes of the affine root system. Then,
the variety $\M_\th(\delta)$ is smooth and the
canonical map $\pi:\ \M_\th(\delta)\to \M_0(\delta)=\C^2/\Ga$
is the minimal resolution of $\C^2/\Ga$.
\end{thm}

\section{Nakajima varieties}\label{sec_M}
\subsection{} We now combine together all the previous constructions.
Thus, we fix a quiver $Q$ and
consider
the quiver $\ddq$, the double of $\dq$. Given any dimension vector
$(\bv,\bw)\in\Z^I\times Z^I$, 
choose a pair of $I$-graded vector spaces $V=\oplus\ii V_i$ and
$W=\oplus\ii W_i$ such that $\dim_I V=\bv$ and $\dim_I W=\bw.$
By definition, we have
\begin{multline*}
\Rep(\ddq,\bv,\bw)=T^*\Rep(\fq,\bv,\bw)\\=
\Rep(Q,\bv)\times\Rep(Q^\op,\bv)\times\Hom_{\C I}(V,W)\times\Hom_{\C I}(W,V).
\end{multline*}

Thus, one may view an element of  $\Rep(\ddq,\bv,\bw)$
as a quadruple $(\bx,\by,\bi,\bj)$, where
$\bx\in \Rep(Q,\bv),$
$\by\in \Rep(Q^\op,\bv),\
\bi\in \Hom_{\C I}(W,V),$ and $\bj\in\Hom_{\C I}(V,W).$

In particular, we find
\beq{mdim}
\dim\Rep(\ddq,\bv,\bw)=A_{\dq}\bv\cdot\bv+2\bv\cdot\bw.
\eeq

The vector space  $\Rep(\ddq,\bv,\bw)$ has  the symplectic structure of a cotangent bundle and
the group
$G_\bv$ acts on $\Rep(\ddq,\bv,\bw)$ by symplectic
automorphisms, via $G_\bv\ni g:\ (\bx,\by,\bi,\bj)\mto (g\bx g\inv,g\by
g\inv,g\ccirc\bi,\bj\ccirc g\inv).$
The associated moment map
is given by
\beq{moment_nak}
\mu:\
\Rep(\ddq,\bv,\bw)\to\g_\bv^*=\g_\bv,\quad
(\bx,\by,\bi,\bj)\mto \sum\ [x,y] + \bi\ccirc\bj\ \in\ \g_\bv.
\eeq
Here, we write
$\bi\ccirc\bj:=\sum\ii\ \bi_i\ccirc\bj_i$ where, for each $i$,
$\bi_i\ccirc\bj_i:\ V_i\to V_i$ is a rank one operator.

For any $\la\in\C^I$ we have
\beq{adhm}
\mu\inv(\la)=\{(\bx,\by,\bi,\bj)\in \Rep(\ddq,\bv,\bw)\mid
[\bx,\by]+\bi\o\bj=\la\}.
\eeq

From now on, whenever we discuss varieties involving
\eqref{adhm}, {\textbf{we will always assume that $\bw\neq0.$}}

\begin{rem}\label{adhmcb} Recall  the quiver $Q^\bw$
introduced by Crawley-Boevey, see \S\ref{cb_frame}.
One can use the isomorphism in
\eqref{pirep} for the quiver $Q^\bw$ to identify the
scheme \eqref{adhm} with $\Rep(\Pi_{\wh\la},\wh\bv)$,
the representation scheme of the preprojective algebra for the
the quiver $Q^\bw$ where $\wh\la$ is an appropriate parameter.
\erem

Given $\th\in\Z^I$, we may apply general Definition
\ref{stab} to the variety $\mu\inv(\la)$ and the character
$\chi_\th$ of the group $G_\bv$. This way, one proves

\begin{prop}\label{stab_nak}  {\em  A quadruple $(\bx,\by,\bi,\bj)\in
\mu\inv(\la)$ is $\th$-semistable if and only if  the following holds:}\vs

\noindent
 For any collection of vector subspaces $S=(S_i)\ii\
\sset\ V=(V_i)\ii$
which is stable under the maps $\bx$ and $\by$, we have
\begin{align}
S_i\sset \Ker \bj_i,\en \forall i\in I\quad&\Longrightarrow\quad
\th(\dim_I S)\leq 0;\label{cond1}\\
S_i\supset \Image \bi_i,\en  \forall i\in I\quad&\Longrightarrow\quad
\th(\dim_I S)\leq\th(\dim_I V).\label{cond2}
\end{align}
\end{prop}

\begin{examp} In the case $\th=0$, any point in $\mu\inv(\la)$ is
$\th$-semistable.
Such a  point is $\th$-stable if and only if the only collection of  subspaces $S=(S_i)\ii\
\sset\ V=(V_i)\ii$
which is stable under the maps $\bx$ and $\by$, is $S=0$ or $S=V$.
\end{examp}

The  above proposition implies, in particular, the following result
\begin{cor}\label{th_cor} In the special case where $\th=\pm\th^+$,
the point $(\bx,\by,\bi,\bj)\in
\mu\inv(\la)$ is $\th$-semistable if and only if, in the notation of Proposition
\ref{stab_nak}$\operatorname{(i)}$,
  we have
\begin{align*}
S_i\sset \Ker \bj_i,\en \forall i\in I\quad&\Longrightarrow\quad
S= 0\quad\text{if}\en\th=\th^+,\en\text{resp.}\\
S_i\supset \Image \bi_i,\en  \forall i\in I\quad&\Longrightarrow\quad
S= V\quad\text{if}\en\th=-\th^+.
\end{align*}
\end{cor}

\begin{defn}\label{def_M}
The variety $\M_{\la,\th}(\bv,\bw):=\mu\inv(\la)^{ss}_\th/_{\chi_\th}
G_\bv$
is called the {\em Nakajima variety} with parameters $(\la,\th)$.
Let $\M^s_{\la,\th}(\bv,\bw)\sset \M_{\la,\th}(\bv,\bw)$ denote the
Zariski open subset corresponding to stable points.
\end{defn}

Thanks to the general formalism of Hamiltonian reduction,
the symplectic structure on the manifold $\Rep(\ddq,\bv,\bw)$
gives $\M_{\la,\th}(\bv,\bw)$ the canonical structure of
a (not necessarily smooth) algebraic Poisson variety.

\begin{rem}
The equation $[\bx,\by]+\bi\o\bj=\la$, in \eqref{adhm}, is often called the
{\em moment map equation}, or
the {\em ADHM-equation}, since an equation of this form was first
considered
by Atiyah, Hitchin, Drinfeld, and Manin in their work
on instantons on ${\mathbb P}^2$, cf. \cite{ADHM} and also \cite{Na3}.

 From that point of view,
it is natural to view  \eqref{adhm} as part of
a larger system of  hyper-K\"ahler moment map equations,
cf. Remark \ref{hyperkahler}. Accordingly,
we will refer to the pair  $(\la,\th)$, viewed as an element of the real vector
space $\BR^3\o \BR^\bv=(\C\oplus\BR)\o \BR^\bv=\C^\bv\oplus \BR^\bv$,
as a   `hyper-K\"ahler parameter'.
\erem

\subsection{}\label{group}  To formulate the main properties of Nakajima
varieties, fix a quiver $Q$ and  write  $C_Q$ for the Cartan
matrix of $Q$. 
We  introduce the following set,
$$R':=\{\bv\in\Z^I\sminus \{0\}\en\;\big|\en\; C_Q\bv\cdot\bv\leq
2,\ \forall i\in I\}.
$$

If $Q$ is a quiver of either finite Dynkin or
extended Dynkin types, then 
$R'_Q=R_Q$ is the set of roots 
 associated with the Cartan matrix $C_Q$.
This is not necessarily true for more general quivers.

For $\al\in\BR^I$, write $\bv^\perp:=\{\la\in\BR^I\mid \la\cdot\bv=0\}$.

Given a dimension vector $\bv\in\Z^I_{\geq 0},$ the  parameter $(\la,\th)\in \C^I\times\Z^I$
 will be called $\bv$-{\em regular}
 if, viewed as a hyper-K\"ahler parameter $(\la,\th)\in \BR^3\o \BR^I$,
 it satisfies, cf. \cite[\S1(iii)]{Na6},
\beq{reg}
(\la,\th)\in  (\BR^3\o \BR^I)\
\sminus\
\bigcup_{\{\al\in R'_Q\,\ |\,\ 0\leq\al\leq\bv\}}\
\BR^3\o \al^\perp.
\eeq

We note that $(\la,\th):=(0,\th^+)$ is a $\bv$-regular parameter
for any dimension vector.

\begin{thm}\label{nak_thm} Fix $\la\in\C^I,\
\th\in\Z^I$, where $\la\cdot\bv=0$.  Then, we have

\vi We have $\M_{\la,0}(\bv,\bw)=
\mu\inv(\la)/\!/G_\bv$; this is an affine variety and
there is a canonical projective morphism
$\pi:\ \M_{\la,\th}(\bv,\bw)\to\M_{\la,0}(\bv,\bw)$,
which respects the Poisson brackets.

\vii Let
the  parameter $(\la,\th)$ be $\bv$-regular.
Then any $\th$-semistable point in $\mu\inv(\la)$ is
$\th$-stable, so 
  $\M_{\la,\th}(\bv,\bw)=\M^s_{\la,\th}(\bv,\bw),$ 
cf. \cite[\S2.8]{Na1},\cite[\S3(ii)]{Na1}. Furthermore,
this variety is 
smooth and connected variety of dimension
$$\dim\M_{\la,\th}(\bv,\bw)\;=\; 2\bw\cdot\bv-C_Q\bv\cdot\bv.$$
The Poisson structure on  $\M_{\la,\th}(\bv,\bw)$ is nondegenerate
making it an algebraic symplectic manifold.

\viii The variety $\M_{0,\th^+}(\bv,\bw)$
contains $T^*\R_{\th^+}(\bv,\bw)$ as a Zariski open subset.
\end{thm}

\begin{proof}[Sketch of Proof] Part (i) is clear. 
To prove (ii), one shows that the isotropy group
of any point  $(\bx,\by,\bi,\bj)\in \mu\inv(\la)$ that
satisfies conditions \eqref{cond1}-\eqref{cond2} is
trivial, provided the parameter $(\la,\th)$ is $\bv$-regular.
It follows, in particular, that the $G_\bv$-orbit
of a semistable point   $(\bx,\by,\bi,\bj)\in \mu\inv(\la)$ 
must
be an orbit of maximal dimension equal to $\dim G_\bv$. We conclude that
one
 semistable orbit
can not be contained in the closure of another semistable orbit.
Thus, all  semistable  orbits are closed in $\mu\inv(\la)^{ss}$,
hence any  semistable point is actually stable.
 
Further, by Corollary \ref{smooth_fiber}, the triviality of stabilizers
implies
that the set $\mu\inv(\la)^{ss}$ of $\th$-stable points  is smooth
and $\mu\inv(\la)^{ss}/G_\bv$ is a symplectic manifold.
Therefore, using the dimension formula \eqref{mdim}, we compute
$$\dim\big(\mu\inv(\la)^{ss}/G_\bv\big)=
2\bw\cdot\bv+(2\Id-C_Q)\bv\cdot\bv-2\dim
G_\bv=
2\bw\cdot\bv-C_Q\bv\cdot\bv.$$
(note that unlike the situation considered in Theorem \ref{Mthm}
the $G_\bv$-action on $\Rep(\ddq, \bv,\bw)$ does not factor
through the quotient $G_\bv/\gm$. Therefore, it is the
dimension of the group $G_\bv$, rather than that of
$G_\bv/\gm$, that enters the dimension count above).

Finally, the connectedness of the varieties $\M_{\la,\th}(\bv,\bw)$ is
a much more difficult result proved by Crawley-Boevey.
The proof is based on the irreducibility statement in Theorem
\ref{gg}(iii)
and on a `hyper-K\"ahler rotation' trick (Remark \ref{rotat} below).
For more details see comments
at \cite{CB1}, p.261.
\end{proof}

\begin{rem}\label{rotat} For a $\bv$-regular parameter $(\la,\th)$, the Nakajima
variety $\M_{\la,\th}(\bv,\bw)$ comes equipped with
a structure of hyper-K\"ahler manifold, cf. Remark \ref{hyperkahler}.
In particular, one can show that there is a choice of complex structure on
the $C^\infty$-manifold
 $\M_{\la,\th}(\bv,\bw)$ that makes it a smooth
and {\em affine} algebraic variety, see \cite[\S\S3.1, 4.2]{Na1}.
\erem

Observe next that the group $G_\bw=\prod\ii GL(w_i)$ acts naturally
on $\Rep(\fq,\bv,\bw)$ and on $\Rep(\ddq,\bv,\bw)$.
Furthermore, the $G_\bw$-action on the latter space is
Hamiltonian and  each fiber $\mu\inv(\la)$ of the moment map
\eqref{moment} is a $G_\bw$-stable subvariety.
Also,  the $G_\bw$-action clearly preserves any stability condition
hence descends, for any  $(\la,\th)$, to a well defined   $G_\bw$-action on
the Nakajima variety $\M_{\la,\th}(\bv,\bw)$ by Poisson automorphisms.

Assume now that $\la=0$. In this special case, there are two natural ways to
define an additional
$\gm$-action on  $\M_{0,\th}(\bv,\bw)$ that makes it
a $G_\bw\times\gm$-variety. Each of these actions comes from
a  $\gm$-action on  $\Rep(\ddq,\bv,\bw)$ that keeps the fiber
$\mu\inv(0)$
stable and commutes with the $G_\bv$-action on the fiber.
The first  $\gm$-action on  $\Rep(\ddq,\bv,\bw)$
is the dilation action given  by the formula
$\gm\ni t:\ (\bx,\by,\bi,\bj)\mto
(t\cdot\bx,t\cdot\by,t\cdot\bi,t\cdot\bj)$. This action rescales the symplectic 2-form
$\omega$ as $t:\ \omega\mto t^2\cdot\omega$.

The second  $\gm$-action on  $\Rep(\ddq,\bv,\bw)$ 
 corresponds, via the identification
$\Rep(\ddq,\bv,\bw)$ $=T^*\Rep(\fq,\bv,\bw)$, to the natural $\gm$-action
by dilations along the fibers of the contangent
bundle.
This $\gm$-action is defined by the formula
$\gm\ni t:\ (\bx,\by,\bi,\bj)\mto(\bx, t\cdot \by,\bi, t\cdot \bj)$.
The latter action keeps the
subvariety $\mu\inv(0)$ stable and commutes with the
$G_\bv$-action on $\Rep(\ddq,\bv,\bw)$. Therefore,
for any $\th$, there is an
induced $G_\bv$-action $\gm\ni t:\
z\mto t(z),$ on $\M_{0,\th}(\bv,\bw)$.
Furthermore, the map
$\pi$ becomes a $G_\bv$-equivariant morphism
of $G_\bv$-varieties, and the fiber
$\pi\inv(0)\sset \M_{0,\th}(\bv,\bw)$
becomes a  $G_\bv$-stable subvariety.

The symplectic form 
$\om$ on $T^*\Rep(\fq,\bv,\bw)$ gets rescaled under the above $\gm$-action
as follows $\gm\ni t:\ \om\mto t\cdot\om$.
Hence, the induced symplectic form on $\M_{0,\th}(\bv,\bw)$,
to be denoted by $\om$ again, transforms in a similar way.

For any parameters $(\la,\th)$, the canonical projective
morphism $\pi: \M_{\la,\th}\to\M_{\la,0}$  is $G_\bw$-equivariant.
In the case $\la=0$ this morphism is also $\gm$-equivariant with
respect to either of the two $\gm$-actions defined above.

\subsection{}\label{strat}
Let $\mu\inv(\la)^\circ\sset \mu\inv(\la)$ be the subset of points with
trivial isotropy group. We let
$\M_{\la,0}^\circ(\bv,\bw)$
$\sset \M_{\la,0}(\bv,\bw)$ be
the image of this set in $\mu\inv(\la)/\!/G_\bv$.
Nakajima uses the notation $\M_{\la,0}^\text{reg}(\bv,\bw)$
for $\M_{\la,0}^\circ(\bv,\bw),$
cf. \cite[\S3(v)]{Na2}.
He verifies
that $\M_{\la,0}^\circ(\bv,\bw)$ is a Zariski open (possibly
empty)
subset of $\M_{\la,0}(\bv,\bw)$.

One has the following result, cf. \cite{Na2}, Proposition 3.24.

\begin{prop}\label{strat_prop} Assume the quiver $Q$ has no edge loops and that $(\la,\th)$ is a
$\bv$-regular parameter. Then, one has
\vs

\vi Any point in $\mu\inv(\la)^\circ$ is $\th$-stable.

\vii If  the set  $\M_{\la,0}^\circ(\bv,\bw)$ is
nonempty then it is dense in $\M_{\la,0}(\bv,\bw)$ and, we have:
\vs

\npb{The canonical projective morphism $\pi: \M_{\la,\th}\to\M_{\la,0}$  is a symplectic
resolution;}
\vs

\npb{The set $\pi\inv\big(\M_{\la,0}^\circ(\bv,\bw)\big)$ is dense
in $\M_{\la,\th}(\bv,\bw)$, and  the map $\pi$ restrics
to an isomorphism}
$$\pi:\
\pi\inv\big(\M_{\la,0}^\circ(\bv,\bw)\big)\
\iso\
\M_{\la,0}^\circ(\bv,\bw).\eqno\Box$$
\end{prop}

There is a combinatorial criterion for the set
$\M_{\la,0}^\circ$ to be nonempty, see \cite[Proposition ~10.5 and
Corollary ~10.8]{Na2}. Also, using Theorem \ref{Zthm}
Nakajima proves, see \cite[Corollary 6.11]{Na1},

\begin{prop}\label{semismall} If the quiver $Q$ has no loop edges then the map
$\pi:\ \M_{\la,\th}(\bv,\bw)\to \pi\big(\M_{\la,0}(\bv,\bw)\big)$
is semismall for any $\bv$-regular parameter $(\la,\th)$.
\end{prop}

\begin{examp}[Type $\textbf{A}$ Dynkin quiver]\label{dyn2}
Let $Q$ be an $\textbf{A}_n$-quiver, and let  $\bv=(v_1,v_2,\ldots,v_n)$ and $\bw=(r,0,0,\ldots,0)$,
where $r>v_1>v_2>\ldots>v_n>0,$
as in Example \ref{dyn}. Thus, a representation of the quiver
$\ddq$ looks like %
$$ \Rep(\ddq,\bv,\bw):\
\xymatrix{\ \stackrel{W_1}\hd\ \ar@<0.5ex>[r]^<>(0.5){\bi}&
\ \stackrel{V_1}\hd\ \ar@<0.5ex>[r]^<>(0.5){\by}\ar@<1ex>[l]^<>(0.5){\bj}&\ \stackrel{V_2}\hd\
\ar@<0.5ex>[r]^<>(0.5){\by}\ar@<1ex>[l]^<>(0.5){\bx}&\ldots
\ar@<0.5ex>[r]^<>(0.5){\by}\ar@<1ex>[l]^<>(0.5){\bx}&\ 
\stackrel{V_{n-2}}\hd\ \ar@<0.5ex>[r]^<>(0.5){\by}\ar@<1ex>[l]^<>(0.5){\bx}&
\ \stackrel{V_{n-1}}\hd\ \ar@<0.5ex>[r]^<>(0.5){\by}
\ar@<1ex>[l]^<>(0.5){\bx}&\ \stackrel{V_n}\hd\ \ar@<1ex>[l]^<>(0.5){\bx}
}
$$

Write $W:=W_1$.
The assignment $(\bx,\by,\bi,\bj)\mto \bj\ccirc\bi$
gives a map $\varpi:\ \Rep(\ddq,\bv,\bw)\to \g{\mathfrak{l}}(W)$.
Let $\varpi$ denote the restriction of this map
to
$\mu\inv(0)\sset \Rep(\ddq,\bv,\bw),$ the zero fiber of the moment map,
and let $X:=\varpi(\mu\inv(0))$ be the image of $\varpi$.

Recall that, according to the discussion 
in  Example \ref{dyn},
 we have $\R_{\th^+}(\bv,\bw)\cong \FF(n,W).$

\begin{prop}\label{kpp} \vi The map $\varpi$ induces an isomorphism
$\M_{0,0}(\bv,\bw)=\mu\inv(0)/\!/G_\bv\iso X$, cf. \cite{Na1} and
\cite[Theorem 2.1]{Sh}.

\vii One has an isomorphism $\M_{0,\th^+}(\bv,\bw)\cong
T^*\R_{\th^+}(\bv,\bw)=T^*\FF(n,W)$ such that the canonical
map $\pi:\ \M_{0,\th^+}(\bv,\bw)\to\M_{0,0}(\bv,\bw)$
gets identified with natural moment map
$T^*\FF(n,W)\to X\sset\g{\mathfrak{l}}(W)$, see \cite[Theorem 7.2]{Na1}.
\end{prop}
Here, the isomorphism  $\M_{0,\th^+}(\bv,\bw)\cong
T^*\R_{\th^+}(\bv,\bw)$, of part (ii), is a particularly nice
case of the situation considered in Theorem \ref{nak_thm}(iii).

Following an observation made by Shmelkin, we alert the reader 
that the map $T^*\R_{\th^+}(\bv,\bw)$ $\to X$ need not 
be surjective, in general. More precisely, one has the
following result, which is essentially
due Kraft and Procesi \cite{KP}, cf.
also \cite[Proposition 2.2(ii)]{Sh} and \cite[\S7]{Na1}:

\begin{prop}\label{kp} In the above setting, assume in addition
that the following (stronger) inequalities hold
\beq{strong}
r-v_1\geq v_1-v_2\geq v_2-v_3\geq \ldots\geq v_{n-1}-v_n\geq v_n.
\eeq

Then,  the map $T^*\R_{\th^+}(\bv,\bw)\to X$
is  surjective. Furthermore, the set $\M_{0,0}^\circ(\bv,\bw)$ gets identified,
under the  isomorphism
$\M_{0,0}(\bv,\bw)\cong X$, with
the unique Zariski open and dense
$GL(W)$-conjugacy class in $X$.
\end{prop}

Thus, in this case the affine variety $\M_{0,0}(\bv,\bw)$
 gets identified with the closure of a nilpotent
$GL(W)$-conjugacy class in $\g{\mathfrak{l}}(W)$.
\end{examp}
\subsection{A Lagrangian subvariety}\label{lagr_sec}
 We recall the following standard

\begin{defn}
A locally closed subvariety $\L$ of a symplectic manifold $(M,\om)$ is called
{\em Lagrangian} if the tangent space to $\L$ at any smooth 
point $x\in \L$ is a maximal isotropic subspace
of $T_xM$ (the tangent space to $M$ at $x$)
with respect to the
symplectic 2-form $\om$.
\end{defn}

From now on, we fix a quiver $Q$, and we let
$\la=0$. Below, we will use the
second of the two
 $\gm$-actions on  Nakajima 
varieties, introduced in \S\ref{group}.
Recall that this  action is given by the formula
$\gm\ni t:\ \phi=(\bx,\by,\bi,\bj)\,\mto\, t(\phi)=(\bx, t\cdot
\by,\bi, t\cdot \bj)$.
Thus, for any $\th\in\Z^I,$ we
have  the canonical $\gm$-equivariant projective morphism
$\pi:\
\M_{0,\th}(\bv,\bw)\to\M_{0,0}(\bv,\bw)$.

We define 
$\L_\th(\bv,\bw):=
[\pi\inv\big(\M_{0,0}(\bv,\bw)^{\gm}\big)]_{\oper{red}},$
the preimage of the $\gm$-fixed point set
equipped with reduced scheme structure.
Thus, 
$\L_\th(\bv,\bw)\sset \M_{0,\th}(\bv,\bw)$
is a reduced closed subscheme.

\begin{thm}\label{lagr} For a $\bv$-regular parameter $(0,\th)$,
we have:\vs

\vi Each irreducible component of the 
variety $\L_\th(\bv,\bw)$ is a Lagrangian subvariety
of $\M_{0,\th}(\bv,\bw)$, a symplectic manifold.\vs

\vii Assume, in addition, that
 $\th=\th^+$ and the quiver  $Q$ has no oriented cycles.
 Then $\L_\th(\bv,\bw)=\pi\inv(0)$; furthermore,
 the   $G_\bv$-orbit of a quadruple
$(\bx,\by,\bi,\bj)\in\mu\inv(0)^{ss}$ represents a point of  $\L_\th(\bv,\bw)$
if and only if  we have $\bi=0$ and 
the $G_\bv$-orbit of the pair $(\bx,\by)\in \Rep(\dq,\bv)$
contains the pair $(0,0)\in \Rep(\dq,\bv)$ in its closure.
\end{thm}

\begin{rem} The statement in  (ii)  motivates the name `nilpotent
variety'
for the variety $\pi\inv(0)$.\qq
\end{rem}

We will now proceed with the proof of Theorem \ref{lagr}(i). 

First of all, observe that for any representation
$\phi=(\bx,\by,\bi,\bj)\in\Rep(\ddq,\bv,\bw)$,
we have
$\underset{^{t\to0}}\lim\,t(\phi)=$
$\underset{^{t\to0}}\lim\,(\bx,t\cdot\by,\bi,t\cdot\bj)=(\bx,0,\bi,0).$ The image of the point
$(\bx,0,\bi,0)$ in $\Rep(\ddq,\bv,\bw)/\!/G_\bv,$ the categorical quotient,
is clearly  a $\gm$-fixed point.
Thus, we conclude that the $\gm$-action provides 
 a contraction of $\M_{0,0}(\bv,\bw)$ to 
$\M_{0,0}(\bv,\bw)^{\gm}$, the fixed point set.

Further, the fixed point set of the $\gm$-action in the smooth variety
$\M_{0,\th}(\bv,\bw)$  is a (necessarily smooth) 
subvariety  $F:=\M_{0,\th}(\bv,\bw)^{\gm}\ \sset \
\M_{0,\th}(\bv,\bw).$
We write $F_1,\ldots,F_r$ for the connected components of $F$,
and introduce the following sets
\beq{lim}
\L_s:=\{z\in \M_{0,\th}(\bv,\bw)\mid 
\underset{^{t\to\infty}}\lim\, t(z)\en\text{exists, and we have}\en
\underset{^{t\to\infty}}\lim\, t(z)\in F_s\},\quad
s=1,\ldots,r.
\eeq

\begin{lem}\label{lag_lem} For any quiver $Q$,
the set $F$ is contained in 
$\pi\inv\big(\M_{0,0}(\bv,\bw)^{\gm}\big),$ and
there is  a decomposition 
$\pi\inv\big(\M_{0,0}(\bv,\bw)^{\gm}\big)=\bigsqcup_{1\leq s\leq r}\ \L_s.$
\end{lem}
\begin{proof} Since $\pi$ is a $\gm$-equivariant morphism,
we have $\pi\big(\M_{0,\th}(\bv,\bw)^{\gm}\big)\sset
\M_{0,0}(\bv,\bw)^{\gm}$. In particular, one has
$F\sset \pi\inv\big(\M_{0,0}(\bv,\bw)^{\gm}\big).$

Now, fix $\wt z\in 
\M_{0,\th}(\bv,\bw)$ and let $z=\pi(\wt z)\in \M_{0,0}(\bv,\bw).$
We consider
the maps $\gm\to\M_{0,0}(\bv,\bw),\ t\mto t(z)$,
resp. $\gm\to\M_{0,\th}(\bv,\bw),\ t\mto t(\wt z)$.
It is
clear that if the limit of $t(z),\ t\to\infty$,
exists then this limit is  a $\gm$-fixed point.

Assume first that
$\wt z\in L_\theta(\bv,\bw).$
Then, $z$ is a $\gm$-fixed point
and $t(\wt z)\in \pi\inv(z)$ for any $t$.
It follows, since $\pi\inv(z)$ is a complete variety,
that the map $t\mto t(\wt z)$  extends to a regular map
${\mathbb P}^1\to\M_{0,\th}(\bv,\bw).$
Thus, for any $\wt z\in \pi\inv\big(\M_{0,0}(\bv,\bw)^{\gm}\big)$,
 the limit of $t(\wt z),\ t\to\infty$,
exists and we have $\underset{^{t\to\infty}}\lim\, t(\wt z)\in F.$
We conclude that
$L_\theta(\bv,\bw)\sset\cup_{1\leq s\leq r}\ \L_s.$

Assume next that $\wt z\in\cup_{1\leq s\leq r}\ \L_s,$
so  we have $\underset{^{t\to\infty}}\lim\, t(\wt z)\in F.$
It follows that  the map  $t\mto \pi(t(\wt z))=t(z)$ also has a
  limit as $t\to\infty$. Therefore, 
the map $\gm\to\M_{0,0}(\bv,\bw),\ t\mto t(z)$
 extends to the point
$t=\infty$. On the other hand,
by the observation made before 
Lemma \ref{lag_lem}, the
latter map automatically
 extends to the point
$t=0$. Therefore, we get a regular map
${\mathbb P}^1\to\M_{0,0}(\bv,\bw)$. Such a map
must be a constant map, since $\M_{0,0}(\bv,\bw)$
is affine. Thus, we must have
$\pi(\wt z)=z\in \M_{0,0}(\bv,\bw)^{\gm}.$
We conclude that
$\wt z\in  L_\theta(\bv,\bw)$.
The result follows.
\end{proof}

\begin{rem} We have shown that
the $\gm$-action provides a
contraction of the variety $\M_{0,\th}(\bv,\bw)$ to
the fixed point set
$F$.
\end{rem}

Theorem \ref{lagr}(i) is clearly a consequence of the
following more precise result
\begin{prop}\label{lagr_prop} Each piece $\L_s$
is a smooth, connected, locally closed Lagrangian subvariety of
$\M_{0,\th}(\bv,\bw)$.

Furthermore, the closures
$\obar{\L}_s,\ s=1,\ldots,r,$ are precisely the irreducible 
components of $\L$.
\end{prop}
\begin{proof}
The pieces defined by equation \eqref{lim} are known
as the Bialinicki-Birula pieces. The Bialinicki-Birula pieces are shown
by  Bialinicki-Birula to be  smooth, connected, and locally closed
subsets, for any $\gm$-action on a smooth quasi-projective
variety $X$, provided the action contracts $X$ to
$X^{\gm}$.
The first statement of the proposition follows.

Next, we fix a connected component
$F_s$ and a point $\phi\in F_s$. The tangent space to
$\M_{0,\th}(\bv,\bw)$ at $\phi$ has a weight decomposition with respect to
the $\gm$-action
\beq{weight}
T_\phi\big(\M_{0,\th}(\bv,\bw)\big)=\bplus_{m\in\Z}\ H_m,
\eeq
such that $t\in\gm$ acts on the direct summand
$H_m$ via multiplication by $t^m$.
 In particular, we see that $H_0=T_\phi F$,
is the tangent space to the fixed
point set $F$.

Recall that the symplectic 2-form $\om$ on
$\M_{0,\th}(\bv,\bw)$ has weight $+1$ with respect to the $\gm$-action.
Hence, a pair of  direct summands
$H_k$ and $H_l$ are $\om$-orthogonal unless $k+l=1$;
furthermore,
the 2-form  gives a perfect pairing $\om:\
H_m\times H_{1-m}\to\C,$
for any $m\in\Z$. We see, in particular, that
$\bplus_{m\leq 0}\ H_m$ is a Lagrangian subspace in $\bplus_{m\in\Z}\
H_m$.

To complete the proof, pick $z\in\L_s$ such that
$\underset{^{t\to\infty}}\lim\, t(z)=\phi.$ 
It is clear that, for the curve  $t\mto t(z)$ to have a limit
as $t\to\infty$,
the tangent vector to the curve 
at $t=\infty$ must belong to the span of nonpositive weight
subspaces. In other words, we must have
$$\frac{d(t(z))}{dt}\big|_{t=\infty}\
\in\ \bplus_{m<0}\ H_m.
$$

Since $\L_s$ is smooth at $\phi$,
we deduce the equation $T_\phi(\L_s)=\bplus_{m\leq 0}\ H_m.$
It follows,
by the above, that $T_\phi(\L_s)$
 is a Lagrangian subspace in 
$T_\phi\big(\M_{0,\th}(\bv,\bw)\big)$, and the first statement
of the proposition is proved.

Now, the decomposition of
Lemma \ref{lag_lem} presents $\L$ as a union of irreducible
varieties of equal dimensions, and
the second statement of the proposition  follows.
\end{proof}

We also prove the following result
which is part of statement (ii) of Theorem \ref{lagr}.

\begin{lem}\label{fixed} 
If the quiver $Q$ has no oriented cycles, then
one has
$\M_{0,0}(\bv,\bw)^{\gm} =\{0\},$ the only fixed point.

Thus, in this case, we have $L_\th(\bv,\bw)=\pi\inv(0).$
\end{lem}
\begin{proof} A  $\gm$-fixed point 
in $\M_{0,0}(\bv,\bw)=\mu\inv(0)/G_\bv$
is represented, by definition, by a $\gm$-stable and closed 
$G_\bv$-orbit  $\O\sset\mu\inv(0)$. Hence, for any point
$(\bx,\by,\bi,\bj)\in \O$, in this closed orbit, we must have
$\underset{^{t\to\infty}}\lim\,(\bx,t\cdot\by,\bi,t\cdot\bj)\in \O$.
Thus, we get $(\bx,0,\bi,0)\in \O$.
We conclude that the orbit $\O$
is contained in $\Rep(\fq,\bv,\bw)$, the zero section of
the cotangent bundle $T^*\Rep(\fq,\bv,\bw)= \Rep(\ddq,\bv,\bw)$.

Observe next that, for any homogeneous polynomial
$f\in \C[\Rep(\fq,\bv,\bw)]^{G_\bv}$, of positive degree, we have
$f|_{\Rep(\fq,\bv,\bw)}=0,$ since $Q$ has no oriented cycles,
see Proposition \ref{BP}. Also, the restriction 
map $\C[\Rep(\fq,\bv,\bw)]^{G_\bv}\to
\C[\mu\inv(0)]^{G_\bv}$ is a surjection,
since $\mu\inv(0)$ is a closed subvariety and the group
$G_\bv$ is reductive. It follows from this that
any homogeneous invariant polynomial  $f\in \C[\mu\inv(0)]^{G_\bv},$
 of positive degree, vanishes on the orbit $\O$.
But $G_\bv$-invariant polynomials are known to separate
closed $G_\bv$-orbits. Thus, $\O=\{0\}.$
\end{proof}

\subsection{Proof of Proposition \ref{wiz}}\label{proofwiz}
The cohomology vanishing
is a standard application of the Grauert-Riemenschneider theorem. The
latter
theorem says that higher derived direct images of the canonical sheaf
under a
proper morphism vanish. We apply this to the proper morphism
$\pi:\ \x\to X$. The canonical sheaf of $\x$ is isomorphic to
$\oo_\x$ since $\x$ is a symplectic manifold. Hence, 
 the Grauert-Riemenschneider theorem yields
$R^i\pi_*\oo_\x=0$ for all $i>0$.
It follows that the complex $R\pi_*\oo_\x$ representing
the  derived direct image is quasi-isomorphic
to $\pi_*\oo_\x$, an ordinary direct image.

Now, using the above, for any $i>0$, we compute
$$H^i(\x,\oo_\x)={\mathbb H}^i(X, R\pi_*\oo_\x)=
H^i(X, \pi_*\oo_\x)=0,$$
where ${\mathbb H}^i(-)$ stands for the hyper-cohomology of a complex of
sheaves
and where the rightmost equality above holds since $\pi_*\oo_\x$ is a coherent sheaf on
an {\em affine} variety.

To prove the second statement of  Proposition \ref{wiz}
one uses the following argument due to  Wierzba.
Write $\om$ for the algebraic symplectic 2-form on $\x$
and let $\overline\om$ be its complex conjugate, an
anti-holomorphic 2-form. This 2-form gives a Dolbeau  cohomology class
$[\overline\om]\in H^2(\x,\oo_\x)$. The latter class is in fact equal
to zero since we have shown that $H^2(\x,\oo_\x)=0.$

Now, let $x\in X$. We must prove that
the restriction of the 2-form $\om$, equivalently,
the  restriction of the 2-form $\overline\om$, to
$\pi\inv(x)$ vanishes. To this end, let
$Y\onto \pi\inv(x)$ be a resolution of singularities of
the fiber, and write $f: Y\to \x$ for the composite
$Y\onto \pi\inv(x)\into \x$. Thus, $f^*\overline\om$
is an anti-holomorphic 2-form on $Y$ and,
in  Dolbeau  cohomology of $Y$, we have
$[f^*\overline\om]=f^*[\overline\om]=0$.

On the other hand, $Y$ is a smooth and projective variety.
Hence, by Hodge theory, we have
$H^2(Y,\oo_Y)\iso H^{0,2}(Y,\C)\sset H^2(Y,\C)$.
It is clear that the Dolbeau  cohomology class of
the 2-form $f^*\overline\om$ goes, under this isomorphism,
to the de Rham cohomology class of  $f^*\overline\om$.
Thus, in de Rham cohomology of $Y$, we have
$[f^*\overline\om]=0$. But any nonzero  anti-holomorphic 
differential form on a K\"ahler manifold gives a nonzero de Rham
cohomology class, thanks to Hodge theory.
It follows that the 2-form $f^*\overline\om$ vanishes,
hence $\overline\om|_{\pi\inv(x)}=0$,
and we are done.\qed
\medskip

We remark that, for a quiver without edge-loops,
the  inequality 
$\dim\pi\inv(0)\leq \frac{1}{2}\dim\M_{0,\th}(\bv,\bw)$
is an immediate consequence of Proposition \ref{wiz}
and Proposition \ref{strat_prop}.

 Here is another approach to the proof of this
inequality (in the special case of quiver varieties). The argument 
below,
based on the `hyper-K\"ahler rotation' trick,
was suggested to me by Nakajima.

In more detail, using a hyper-K\"ahler rotation, cf. Remark \ref{rotat}, we may view
$\M_{0,\th}(\bv,\bw)$ as a smooth and affine
algebraic variety of complex dimension $2d$, say.
It follows 
that singular homology groups, $H_i(\M_{0,\th}(\bv,\bw),\ \BR)$,
vanish for all $i>2d,$
by a standard argument from Morse theory.
On the other hand, $\L_\th(\bv,\bw)$ is a compact subset
of $\M_{0,\th}(\bv,\bw).$ Hence, each irreducible component
of $\L_\th(\bv,\bw)$ of complex dimension $n$ (in the original complex
structure)
gives a nonzero homology class in $H_{2n}(\M_{0,\th}(\bv,\bw),\ \BR)$.
Thus, we must have $2n\leq 2d,$ and the  required dimension inequality 
follows.

\subsection{Hilbert scheme of points}\label{hilb}
Let $Q$ be the Jordan quiver, and let $\bv\in\Z$ be a positive integer.

 In the setting of Example 
\ref{comm_examp1}, the fiber of the moment map over
a central element $\la\cdot\Id\in\gl_\bv$ equals
$$\mu\inv(\la\cdot\Id)=\{(\bx,\by)\in\gl_\bv\times\gl_\bv\mid
[\bx,\by]=\la\cdot\Id\}.$$
This variety is empty for $\la\neq 0$ since we have
$\Tr([\bx,\by])=0$.
For $\la=0$, we get $\mu\inv(0)={\mathcal Z},$
the {\em commuting variety}
of the Lie algebra $\gl_\bv.$

Let $\imath: \C^\bv\into\gl_\bv$ be the imbedding of diagonal
matrices. Since any two diagonal matrices commute, we get
a closed imbedding
$\imath\times\imath:\  \C^\bv\times\C^\bv\into {\mathcal Z}.$
The group $\BS_\bv\sset GL_\bv,$ of
permutation matrices, acts diagonally on ${\mathcal Z}\sset \gl_\bv\times\gl_\bv,$
 and clearly preserves the
image of the map
$\imath\times\imath.$
Therefore, restriction of $\Ad GL_\bv$-invariant functions induces
an algebra map $(\imath\times\imath)^*:\
\C[{\mathcal Z}]^{\Ad GL_\bv}\to \C[\C^\bv\times\C^\bv]^{\BS_\bv}$.
The latter map can be shown to be an algebra isomorphism.

Thus, we deduce
\beq{Z/G}
\M_{0,0}(\bv)=\mu\inv(0)/\!/GL_\bv=\Spec\C[{\mathcal Z}]^{\Ad GL_\bv}
=\Spec\C[\C^\bv\times\C^\bv]^{\BS_\bv}=
(\C^\bv\times\C^\bv)/\BS_n.
\eeq

Next, we study Nakajima varieties $\M_{\la,\th}(\bv,\bw)$ for the Jordan
quiver $Q$. We have
$$\ddq\en=\quad
\xymatrix{
{\;{\C^\bv}^{\;}}\ar@(ul,ur)^<>(0.5){\bx}\ar@(dl,dr)_<>(0.5){\by}\ar@/^/[rr]^<>(0.5){\bj}&&\en\C^\bw
\en\ar@/^/[ll]^<>(0.5){\bi}
}
$$

Therefore, writing $M_\la(\bv,\bw):=\mu\inv(\la\cdot\Id)$ for the corresponding
fiber of the moment map, we get
$$
M_\la(\bv,\bw)=\{(\bx,\by,\bi,\bj)\in\gl_\bv\times\gl_\bv
\times\Hom(\C^\bw,\C^\bv)\times\Hom(\C^\bv,\C^\bw)\mid
[\bx,\by]+\bi\o\bj=\la\cdot\Id\}.
$$
Here, $\bi\o\bj$ denotes a rank one linear operator
$\C^\bv\to\C^\bv,\ u\mto \langle \bj,u\rangle\cdot\bi$.

The above variety $
M_\la(\bv,\bw)$ is nonempty for any $\la\in\C$. Below, we restrict
ourselves
to the special case $\bw=1$. In this case, we may view $\bi$ as
a vector in $\C^\bv=\Hom(\C,\C^\bv),$ resp.
$\bj$ as a covector in $(\C^\bv)^*=\Hom(\C^\bv,\C)$.

Assume first that $\la\neq0$.
Then, one proves that there is no proper subspace
$0\neq S\not\subseteq\C^\bv$ such that          $\bi\in S$ and such
that $S$ is stable under
the maps $\bx,\by$.
It follows, by Corollary \ref{U/G}(i) that $M_\la(\bv,1)$ is a smooth
affine
variety and
that
the $GL_\bv$-action on this variety is free.
Therefore, each  $GL_\bv$-orbit in   $M_\la(\bv,1)$ is
closed. We conclude 
$$\M_{\la,0}(\bv,1)=M_\la(\bv,1)/\!/GL_\bv
=M_\la(\bv,1)/GL_\bv\,
=:\,
\text{Calogero-Moser
space}.
$$

Also, we compute
$$\dim\M_{\la,0}(\bv,1)=\dim M_\la(\bv,1)-
\dim GL_\bv =(2\bv^2 +2\bv-\bv^2)-\bv^2=2\bv.
$$

Next, let $\la=0$. Then, from the equation
$[\bx,\by]+\bi\o\bj=0$ we deduce 
$$\langle
\bj,\bi\rangle=\Tr(\bi\o\bj)=-\Tr([\bx,\by])=0.$$

It follows that $\bi\o\bj$ is a nilpotent rank one linear 
operator in $\C^\bv$.

The following  result of linear algebra will play an important role in
our analysis, see \cite[Lemma ~12.7]{EG}.

\begin{lem}\label{rud}
Let $\bx,\by\in\gl_\bv$ be a pair of  linear operators such that
$[\bx,\by]$ is a nilpotent rank one operator. Then there exists a basis of
$\C^\bv$
such 
the matrices of $\bx$ and $\by$ in that basis are both
upper-triangular.\qed
\end{lem}

In the case  $\la=0$, the variety  $M_0(\bv,1)$ is neither smooth nor
irreducible. Thus, to get a good quotient  one has to impose a
stability
condition. First, let $\th=0$, so $\M_{0,0}(\bv,1)$ is an affine
algebraic variety, by Theorem \ref{nak_thm}.

To describe this variety explicitly, one  shows using the above lemma, that the
assignment sending a quadruple $(\bx,\by,\bi,\bj)\in M_0(\bv,1)$ 
to the joint spectrum $(\Spec \bx,\Spec\by)\in\C^\bv\times\C^\bv$,
of the operators $\bx$ and $\by$, written in an upper-triangular
form provided by lemma \ref{rud}, gives a well-defined morphism
$M_0(\bv,1)\to (\C^\bv\times\C^\bv)/\BS_\bv$,
of algebraic varieties. Moreover, 
this morphism turns out to induce an algebra isomorphism
$\C[\C^\bv\times\C^\bv]^{\BS_\bv}\iso\C[M_0(\bv,1)]^{G_\bv}$.

We conclude that the Nakajima variety
 with parameters $(\la,\th)=(0,0)$
is an affine variety
\beq{M0}
\M_{0,0}(\bv,1)\cong(\C^\bv\times\C^\bv)/\BS_\bv.
\eeq

\begin{rem} The $\gm$-action on
$\M_{0,0}(\bv,1)$ that has been defined
in the previous subsection goes under the above isomorphsm
to a $\gm$-action on
$(\C^\bv\times\C^\bv)/\BS_\bv$. The latter action is given by
$\gm\ni t:\ (u,v)\mto (u,t\cdot v)$. The fixed points
of that action form a subset
$(\C^\bv/\BS_\bv)\times\{0\}\sset(\C^\bv\times\C^\bv)/\BS_\bv$.
Note the subset in question does not reduce to a single point.
Indeed, the Jordan quiver
has an edge loop and, therefore, Lemma \ref{fixed} does not
apply in our present situation.\qq
\end{rem}

Next, we take $\th:=-\th^+=-1\in\Z$.
With this choice of $\th$, a point  $(\bx,\by,\bi,\bj)\in M_0(\bv,1)$ is
stable if and only if  condition \eqref{cond2} holds, and we have

\begin{prop}\label{hilb_stab} The set of $\th$-semistable points
equals  
\beq{triple}M_0(\bv,1)^s_{-\th^+}=\{(\bx,\by,\bi,\bj)\mid
[\bx,\by]=0,\
\bj=0,\ \text{$\bi$ is a cyclic vector for $(\bx,\by)$}\}.
\eeq
\end{prop}
\begin{proof} According to Theorem \ref{stab_nak},
 a point  $(\bx,\by,\bi,\bj)\in M_0(\bv,1)$ is
stable if and only if  condition \eqref{cond2} holds.
The condition means that $\bi$ is a cyclic vector for $(\bx,\by)$,
i.e., we have $\C\langle\bx,\by\rangle\bi=\C^\bv$.

We claim that the last equation implies $\bj=0$.
To see this, we observe that for any $a\in\gl_\bv$, we have
\beq{tr}
\langle \bj,\
a\bi\rangle=\Tr\big(a\ccirc(\bi\o\bj)\big)=-\Tr\big(a\ccirc[\bx,\by]\big),
\qquad\forall a\in\gl_\bv.
\eeq

Assume now that $a\in\C\langle\bx,\by\rangle$, is a noncommutative
polynomial in $\bx$ and $\by$. Then, we may write the matrices
$\bx,\by,$ and $a$  in an upper-triangular form, by Lemma \ref{rud}.
In this form, the principal diagonal of the matrix
$[\bx,\by]$ vanishes, and we get $\Tr\big(a\ccirc[\bx,\by]\big)=0$.
Thus, \eqref{tr} implies that the linear
function $\bj$ vanishes on the vector space
$\C\langle\bx,\by\rangle\bi=\C^\bv$, and the proposition follows. 
\end{proof}

For any commuting  pair $(\bx,\by)\in\gl_\bv$ and any vector
$\bi\in\C^\bv$, we introduce a set of polynomials
in two indeterminates, $x$ and $y$, as follows
$$J_{\bx,\by,\bi}:=\{f\in\C[x,y]\mid
f(\bx,\by)\bi=0\}.
$$

It is clear that $J_{\bx,\by,\bi}$ is an ideal of the algebra
$\C[x,y]$. Furthermore, this  ideal has codimension $\bv$ in $\C[x,y]$ if and only if 
the map $\C[x,y]/J_{\bx,\by,\bi}\to\C^\bv,\ f\mto f(\bx,\by)\bi,$
is surjective. The latter holds if and only if  $\bi$ is a cyclic vector
for the pair $(\bx,\by)$. In fact, one proves

\begin{cor} The assigment $(\bx,\by,\bi)\mto J_{\bx,\by,\bi}$
establishes a bijection between the orbit
set $M_0(\bv,1)^s_{-\th^+}/GL_\bv$ and
the set of   ideals
$J\sset \C[x,y]$ such that $\dim \C[x,y]/J=\bv$.\qed
\end{cor}

The set of   codimension $\bv$ ideals in the algebra
$\C[x,y]$ has a natural scheme structure. The resulting scheme
$\Hilb^n(\C^2)$ 
turns out to be a smooth connected variety of dimension $2n$,
called the {\em Hilbert scheme} of $n$ points in the
plane. Thus, we see that, for $\bw=1$ and $\la=0,\ \th=-\th^+,$ one has
a natural  isomorphism
$$\M_{0,-\th^+}(\bv,1)\cong\Hilb^n(\C^2).$$

In this case, the canonical projective morphism $\pi$, cf. \eqref{M0},
$$
\pi:\
\M_{0,-\th^+}(\bv,1)=\Hilb^n(\C^2)\en\too\en
\M_{0,0}(\bv,1)=(\C^\bv\times\C^\bv)/\BS_\bv,
$$
turns out to be a resolution of singularities,
called the {\em Hilbert-Chow morphism}.

\begin{rem} One can show that
changing our choice of stability condition from $\th=-\th^+$
to  $\th=\th^+$ leads to isomorphic quiver varieties,
because of the isomorphisms 
of Remark \ref{opp}.
\end{rem}

\section{Convolution in homology}\label{cg}
In this section, we review a machinery that produces
associative, not necessarily commutative, algebras
from certain geometric data. The algebras in question
are realized as either homology or $K$-groups of an
appropriate variety, and the corresponding
algebra structure is given by an operation on
homology, resp. on $K$-theory, known as `convolution'.

We refer the reader  to \cite{CG}, ch. 2 and 5, see also  \cite{Gi1},
for more information about the convolution  operation
and for other applications of this construction in
representation theory.

In the next section, the formalism developed below will be applied to quiver
varieties.

\subsection{Convolution}
Let $\C[X]$ denote the vector space of $\C$-valued functions on a
finite set $X$. Characteristic functions of one element
subsets form a $\C$-base of  $\C[X]$.

Let $X_r,\ i=1,2,$ be a pair of finite sets. A linear operator
$K: \C[X_1]\to \C[X_2]$ is given, in the bases of characteristic
functions, by a rectangular $|X_1|\times |X_2|$-matrix
$|K(x_2,x_1)|_{x_i\in X_i}.$ 
We may view this matrix as a  $\C$-valued function
$(x_1,x_2)\mto K(x_1,x_2),$ on $X_1\times X_2$,
called the {\em kernel of the operator} $K$.

The action of $K$ is then given, in terms
of that kernel,
by the formula
\beq{conv1}K:\ f\mto K*f,\quad\text{where}\quad
(K*f)(x_2):=\sum_{x_1\in X_1}\ K(x_2,x_1)\cd f(x_1).
\eeq

Now, let  $X_i,\ i=1,2,3,$ be a triple of finite sets, 
and let  $K: \C[X_1]\to \C[X_2]$ and $K':\
\C[X_2]\to \C[X_3]$ be a pair of operators, with kernels
$K_{32}\in \C[X_3\times X_2]$ and $K_{21}\in
\C[X_2\times X_1]$, respectively.
One may form the composite operator
$K\ccirc K':\ \C[X_1]\to  \C[X_3],\ f\mto K(K'(f))$.

Explicitly, in terms of the kernels, for any $f\in\C[X_1],$ the function
$K(K'(f))$ is given by
\begin{multline*}
x_3\ \mto \ K(K'(f))(x_3)=
\sum_{x_2\in X_2}\ K_{32}(x_3,x_2)\cd
\left(\sum_{x_1\in X_1}\ K_{21}(x_2,x_1)\cd f(x_1)\right)\\
=
\sum_{x_1\in X_1}\ \left(\sum_{x_2\in X_2}\
 K_{32}(x_3,x_2)\cd K_{21}(x_2,x_1)\right)\cdot f(x_1).
\end{multline*}

Thus, the kernel of the composite operator $K\ccirc K'$ is
a function $K_{32}*K_{21}$, on $X_3\times X_1$,
given by the formula
\beq{conv2}
(x_3,x_1)\ \mto\
(K_{32}*K_{21})(x_3,x_1)\ :=\
\sum_{x_2\in X_2}\
 K_{32}(x_3,x_2)\cd K_{21}(x_2,x_1).
\eeq

The operation 
\beq{conv3}
*:\en
\C[X_3\times X_2]\times\C[X_2\times X_1]\too
\C[X_3\times X_1],\quad
K_{32}\times K_{21}\mto K_{32}*K_{21}
\eeq
is called {\em convolution} of kernels. Thinking of 
kernels as of rectangular matrices, the convolution
becomes nothing but matrix multiplication.
Thus, formula \eqref{conv2}
corresponds to the standard matrix multiplication for
$|X_3|\times |X_2|$-matrix by a  $|X_2|\times|X_1|$-matrices.
So, all we have done so far was a reinterpretation of
the fact that composition of linear operators corresponds
to a product of corresponding matrices.

\begin{rem}\label{conv_rem} A
There is an equivalent, but slightly more elegant, way to write
formula \eqref{conv2}  as follows.

For any map $p:\ X\to Y,$ of finite sets, one has
a pull-back map $p^*:\ \C[Y]\to\C[X]$, of functions given by
$(p^*f)(x):=p(f(x)),\ \forall x\in X$.
We also define
a {\em push-forward} linear map on functions by
\beq{push}
p_*:\
\C[X]\to\C[Y],\quad
f\mto p_*f,
\quad\text{where}
\quad(p_*f)(y):=\sum_{\{x\in p\inv(y)\}}\ f(x).
\eeq

For any pair $i,j\in\{1,2,3\},$ let
$p_{ij}:\
X_3\times X_2\times X_1\to X_i\times X_j
$
be the projection along the factor not named.
It is clear that, with the above notation, formula  \eqref{conv2}
may be rewritten as follows
\beq{conv}
K_{32}*K_{21}\ :=\
 (p_{31})_*\,\big((p_{32}^*K_{32})\cdot (p_{21}^*K_{21})\big).
\eeq
\end{rem}

We will be especially interested in a special case of convolution 
\eqref{conv} where
 $X_1=X_2=X_3=X$ is a set with $n$ elements. Then, the convolution product
\eqref{conv} makes $\C[X\times X]$ an 
associative algebra. According to
the preceeding discussion, this algebra is isomorphic to
the algebra of $n\times n$-matrices.

One may  get more interesting examples of convolution algebras by considering
an equivariant version of the above construction, where there is a group
$G$ acting on a finite set $X$. We let $G$ act
diagonally on $X\times X$ and let 
 $\C[X\times X]^G\sset  \C[X\times X]$ be the subspace
of $G$-invariant functions.
This space is clearly isomorphic to
$\C[(X\times X)/G]$, the space of functions on the
set of $G$-diagonal orbits in $X\times X$.

It is immediate to check that  the convolution product
\eqref{conv2}-\eqref{conv3} is $G$-equivariant, hence it
makes  $\C[X\times X]^G$ a subalgebra of
$\C[X\times X]$. 
The resulting
algebra $\bigl\C[X\times X]^G,\ *\bigr$ may be shown to be always
semisimple. Such an algebra need not be simple, so it is not necessarily
isomorphic to a matrix algebra, in general.

\begin{examp}[Group algebra] Given a finite group $G$,
we take $X=G$. We let $G$ act on $X$ by left translations,
and act diagonally on $G\times G$, as before.
Observe that the map $G\times G\to G,\
(g_1,g_2)\mto g\inv_1\cdot g_2$ descends to a well
defined map
$(G\times G)/G\to G$. Moreover, the latter map
is easily seen to be a bijection.

We deduce the following chain of vector space isomorphisms
\beq{iso1}
\C[G\times G]^G\iso \C[(G\times G)/G]\iso\C[G].
\eeq

It is straightforward to check that the restriction 
of convolution \eqref{conv2}-\eqref{conv3}
to $\C[G\times G]^G$ goes, under the 
 composite isomorphism in \eqref{iso1},
to the standard convolution on a group. The latter is given by
$$(f*f')(g)=\sum_{h\in G}\ f(gh\inv)\cd f'(h),
\qquad\forall f,f'\in\C[G].
$$

We conclude that the algebra  $\bigl\C[G\times G]^G,\ *\bigr)$,
with  convolution product
\eqref{conv}, is isomorphic to the
{\em group algebra} of  $G$.
\end{examp}

\begin{examp}[Hecke algebra]
Let $G=G(\F)$ be a split reductive group over
a finite field $\F=\F_q$. Let $B\sset G$ be a Borel
subgroup of $G$. We put $X:=G/B$, and let
$G$ act on $X$ by left translations.
It is known, thanks to the Bruhat decomposition, that
$G$-diagonal orbits in $G/B\times G/B$
are labelled by the elements of
$W$, the Weyl group of $G$.

The resulting
convolution algebra $H_q(G):=\big(\C[G/B\times G/B]^G,\ *\big)$
is called the {\em Hecke algebra} of $G$.
\end{examp}

\subsection{Borel-Moore homology}
We are going to extend the
constructions of the previous subsection to the case where
finite sets are replaced by smooth $C^\infty$-manifolds.

Thus, we  let  $X_i,\ i=1,2,3,$ be a triple of 
smooth manifolds. One might try to replace the summation  in
formula \eqref{conv2}  by integration to get
a convolution product of the form
$\dis
*:\
C^\infty(X_3\times X_2)\times C^\infty(X_2\times X_1)\to
C^\infty(X_3\times X_1)$, cf \eqref{conv3}.

To make this work, one still needs additional ingredients.
One such ingredient is  a {\em measure} on $X_2$
that is necessary in order to define the integral
that replaces summation  in
formula \eqref{conv2}. 

An alternate approach, 
that does not require introducing a measure, is
to replace functions by differential forms.
In this way, one defines a convolution product
\beq{conv5}
\Om^p(X_3\times X_2)\times \Om^q(X_2\times X_1)\to
\Om^{p+q-\dim X_2}(X_3\times X_1),\en
K_{32}\times K_{21}\mto \int_{X_2}
(p_{32}^*K_{32})\wedge (p_{21}^*K_{21}).
\eeq

To insure the convergence of the integral in \eqref{conv5}
one may  assume, for instance, that the manifold $X_2$ is compact.
A slightly weaker assumption,
that is sufficient for \eqref{conv5} to make sence,
is to restrict considerations to
 differential forms with certain  support condition that would
insure, in particular, that the set
\beq{supp}
p_{32}\inv(\supp K_{32})\ \cap\ p_{21}\inv(\supp K_{21})\en
\text{be compact.}
\eeq 

Unfortunately, none of the above works 
in the examples arising from quiver varieties that we would like
 to consider below. In those examples, the manifolds
$X_i,\ i=1,2,3,$ are the quiver varieties, which are {\em noncompact}
complex algebraic varieties. It turns out that
the only natural support condition one could make in those
cases in order
for \eqref{supp} to hold, is to require supports
of $K_{32}$ and $K_{21}$, in  \eqref{conv5}, be contained in appropriate 
{\em closed} algebraic subvarieties. 

Obviously, any  $C^\infty$-differential
form on a manifold whose support is contained
in a closed (proper) submanifold must vanish identically. There are, however,
plenty of `distribution-like' differential forms, called  {\em currents},
which may be supported on closed  submanifolds.  Indeed, replacing 
differential forms by currents 
resolves the convergence problem for integration.
Unfortunately,  introducing currents creates another
problem: the wedge-product operation, which is used in \eqref{conv5},
is not well defined for currents.

All  the above difficulties  may be resolved
by introducing homology. Recall that there is the
 de Rham differential acting on the (graded) vector space
$\Om^\hdot(X)$,
of differential forms
on a manifold $X$.
The homology of the resulting de Rham complex
$\bigl\Om^\hdot(X),\ d\bigr$
is isomorphic to $H^\hdot(X,\C)$, the singular cohomology
of $X$ with complex coefficients.
Similarly, there is a natural  de Rham differential
 on the (graded) vector space
of currents on $X$, and the homology
of the resulting complex is known to be
isomorphic to $H^{BM}_\idot(X,\C)$,
the {\em Borel-Moore homology} of $X$
  with complex coefficients.
The latter is the homology theory that we are going to use.

For practical purposes, it is more convenient to
use a different ({\em a posteriori} equivalent) definition of
Borel-Moore homology
based on Poincar\'e duality rather than on the de Rham complex of currents.
We now recall this definition.

Let $M$ be a smooth {\em oriented} 
$C^\infty$-manifold of real dimension $m$. One defines  Borel-Moore homology
of a closed subset $X\sset M$ to be the following relative cohomology
\beq{BM}H^{BM}_\idot(X):= H^\hdot(M, M\sminus X;\C).
\eeq

It can be shown that the group on the right is, in fact,
independent of the choice of a closed imbedding of $X$ 
into a smooth manifold.

\begin{notation} {\textbf{From now on, we drop the superscript `BM' and
let $H_\idot(X)$ stand for Borel-Moore homology (rather than ordinary
homology)
of $X$.}}
\end{notation}

A property that makes Borel-Moore homology so useful for our purposes is
 that, for any $X$, which is either a smooth connected, and {\em oriented} 
$C^\infty$-manifold or an
 irreducible {\em complex} algebraic variety,
the space $H_m(X),$ where $ m:=\dim_\BR X,$ is 1-dimensional;
furthermore, there is a canonical base element
$[X]\in H_m(X)$,
called the {\em fundamental class} of $X$.

\begin{rem} Note that, in the ordinary homology theory,
fundamental classes  only exist for
{\em compact} manifolds, while  such a compactness condition is
not necessary for the fundamental class to exist in
Borel-Moore homology.\qq
\end{rem}

We record a few basic properties of the Borel-Moore homology theory.
First, for any {\em proper} map $p:  X\to Y$, there is a push-forward
functor $p_*:\ H_\idot(X)\to H_\idot(Y).$

Second, there is a cap-product on  Borel-Moore homology.
In more detail, given two closed subsets $X,Y\sset M,$ where
$M$ is a smooth oriented manifold of real dimension $m$,
one has a cup product
$$\cup:\
H^{m-i}(M,M\sminus X;\ \C)\times
H^{m-j}(M,M\sminus Y;\ \C)\to H^{2m-i-j}(M,M\sminus (X\cup Y);\ \C).
$$

We define a cap-product on  Borel-Moore homology by 
transporting the above cup product  via formula \eqref{BM};
this way we obtain a cap-product pairing
\beq{cap}\cap:\
H_i(X)\times H_j(Y)\to H_{i+j-m}(X),\qquad m=\dim_\BR M.
\eeq
It should be emphasized that the cap-product so defined
{\em does} depend on the ambient smooth
manifold $M$.

\subsection{Convolution  in Borel-Moore homology}\label{conv_sec}
There is  a convolution product in Borel-Moore
homology that provides an adequate generalization,
from the  case of finite sets to the case of
manifolds,
 of the convolution product \eqref{conv}.

To define the convolution product , fix
 $M_i,\ i=1,2,3,$  a triple of smooth oriented
manifolds, and let $p_{ij}:\
M_1\times M_2\times M_3\to M_i\times M_j$ denote the projection along the
factor not named,
cf. \eqref{conv}.

\begin{defn}\label{corr} A pair of
 closed subsets  $Z_{12}\sset M_1\times M_2$ and $Z_{23}\sset
 M_2\times M_3$ is said to be {\em composable} if the following
map \eqref{proper}  is {\em proper}
\beq{proper}
p_{13}:\
(p\inv_{12}Z_{12})\,\cap\,(p\inv_{23} Z_{23})\to M_1\times M_3.
\eeq

Given composable subsets as above, we define their  {\em composite}
to be
$$Z_{12}\circ Z_{23}:=p_{13}\big[(p\inv_{12}Z_{12})\,\cap\,(p\inv_{23}
Z_{23})\big]
\en\sset\en M_1\times M_3.$$
\end{defn}

Now, let
  $Z_{12}\sset M_1\times M_2$ and $Z_{23}\sset
 M_2\times M_3$ be as above, and put $m_i:=\dim M_i$.

We use $M:=M_1\times M_2\times M_3$ as an ambient manifold and  apply formula \eqref{cap}.
In this way,
we get a cap product map
$$ 
\cap:\ 
H_{i+m_3}(p\inv_{12}Z_{12})\times H_{j+m_1}(p\inv_{23}Z_{23})\too
H_{i+j-m_2}((p\inv_{12}Z_{12})\,\cap\,(p\inv_{23} Z_{23})).
$$

Assume further that  $Z_{12}$ and $Z_{23}$ are composable.
Then,
we have a push-forward morphism $(p_{13})_*$, on Borel-Moore homology,  induced by
the {\em proper} map \eqref{proper}.

One defines the 
convolution in Borel-Moore homology as the following map, cf. \eqref{conv}, \eqref{conv5},
\begin{multline}\label{star}
*:\en
H_i(Z_{12}) \times H_j(Z_{23})\
\too\ H_{i+j-\dim M_2}(Z_{12}\circ Z_{23}),\\
c_{12}\times c_{23}\en \mto\en c_{12}* c_{23}:=
(p_{13})_*\Big((c_{12}\boxtimes[M_3])\ \cap\
([M_1]\boxtimes
  c_{23})\Big).
\end{multline}

\subsection{Convolution algebra}\label{sec_b}
Fix $M$, a smooth complex algebraic variety, not necessarily connected,
in general. Further, let  $Y$ be a (not necessarily smooth) algebraic
variety
and 
$\pi: M \to Y$, a {\em proper} morphism.
Thus, we may form a fiber product $Z:=M\times_{Y}M,$ a closed
subvariety of $M\times M$.

One may apply the convolution in Borel-Moore homology operation
in a  special
case where   $M_1=M_2=M_3=M$, and
 $Z_{12}=Z_{23}=Z$.
The assumption the morphism $\pi$ be proper insures that
the set $Z$ is composable with itself in the sense of Definition
\ref{corr}. Furthermore, it is immediate to check that one has
$Z\circ Z=Z$. Thus, the convolution product
\eqref{star} gives $H_\idot(Z)$, the total Borel-Moore homology group
of $Z$, a structure of associative algebra. The
fundamental class $[\Delta]$, of the diagonal
$\Delta\sset M\times M$, is the unit of the algebra
 $\bigl H_\idot(Z),\ *\bigr$.

Next, pick a point $y\in Y$ and put $M_y:=\pi\inv(y).$
Consider the setting of section \ref{conv_sec} in the
special case where $M_1=M_2=M$, and where $M_3=pt$ is a point.
Thus, we have $M_2\times M_3=M_2\times pt=M$, and  put
$Z_{12}:=M\times_YM=Z$, as before,
and $Z_{23}:=M_y=\pi\inv(y),$ viewed as
a closed subset in $M_2\times M_3=M$. 

It is immediate to check that the sets
$Z$ and  $M_y$ are composable and, moreover,
one has $Z\circ M_y=M_y$. Therefore, convolution
in BM homology gives
the space $H_\idot(M_y)$ an $H_\idot(Z)$-module
structure.

Let $\V$ denote a set that provides a labelling for  connectected components of the manifold $M$.
We write $M^{(r)}$ for the  connectected component with label $r\in\V$.
For any pair
$M^{(r)},M^{(s)}$, of connectected components,
we put $Z^{(r,s)}:= Z\cap (M^{(r)}\times M^{(s)}).$ 
Similarly,  we put $M_y^{(r)}:=M_y\cap M^{(r)},$ for any $r\in\V$.
Clearly, we have $H_\idot(Z)=\bplus_{r,s\in\V}\ H_\idot(Z^{(r,s)}),$
resp. $H_\idot(M_y)=\bplus_{r\in\V}\ H_\idot(M_y^{(r)}).$

We write $H\top(M_y^{(r)})$
for the top Borel-Moore homology group of $M_y^{(r)}$.
 This group 
has a  natural basis formed by the fundamental classes 
of irreducible components 
 of the variety $M_y^{(r)}$ of maximal
dimension.

Next, for each pair $(r,s)$, we  introduce a new $\Z$-grading
on the vector space $H_\idot(Z^{(r,s)})$ as follows
\beq{grad} H_{[i]}(Z^{(r,s)}):=H_{d-i}(Z^{(r,s)})
\quad\text{where}\quad d:=\half(\dim_\BR M^{(r)} +\dim_\BR M^{(s)}).
\eeq
We extend this grading to $H_\idot(Z)$
by setting $H_{[i]}(Z)=\bplus_{r,s\in\V}\ H_{[i]}(Z^{(r,s)}).$

The following result is an
immediate consequence of formula \eqref{star}.

\begin{lem}\label{top} \vi The new grading  makes $H_{[\bullet]}(Z)$
a graded algebra with respect to the convolution product,
i.e., we have $H_{[i]}(Z)*H_{[j]}(Z)\sset H_{[i+j]}(Z)$,
for any $i,j\in \Z$. In particular,
$H_{[0]}(Z)$ is a subalgebra  of the  convolution 
 algebra
 $\bigl H_{[\bullet]}(Z),\ *\bigr$.

\vii  For any $y\in Y$, the vector space $H\top(M_y):=\bplus_{r\in\V}\
H\top(M_y^{(r)})$
is stable under the convolution-action of the subalgebra $H_{[0]}(Z)\sset
H_{[\bullet]}(Z)$ on $H_\idot(M_y).$
\qed
\end{lem}

\begin{rem} In the especially important case where $M$ is connected
and the map $\pi: M\to Y$ is semismall, eg. the case
where $\pi$ is a symplectic resolution, for the integer
$d$ appearing in \eqref{grad},
we have $d=\dim_\BR M=\dim_\BR Z.$  Thus, in such a case,
one has $H_{[0]}(Z)=H\top(Z)$. This  group 
has a  natural basis formed by the fundamental classes 
of irreducible components of the variety $Z$
 of maximal
dimension.
\erem

\begin{rem} The material of \S\S\ref{conv_sec}-\ref{sec_b}
is taken from  \cite{Gi3}.
The general notion of convolution algebra in 
Borel-Moore homology,
as well as the geometric construction of its irreducible
representations, was discovered in that paper.
\end{rem}
\section{Kac-Moody algebras and Quiver varieties}\label{km}
\subsection{}\label{st0} 
Throughout this subsection,  we fix  a quiver $Q$, without edge loops,
and a dimension vector $\bw\in\Z^I$.
We also fix a stability parameter $\th\in\BR^I$ and use simplified
notation $\M_\th(\bv,\bw):=\M_{0,\th}(\bv,\bw),$ resp.
$\M_0(\bv,\bw)=\M_{0,0}(\bv,\bw)$.
Recall that we write $\bv\geq\bv'$ whenever $\bv-\bv'\in\Z^I_{\geq 0}$.

Given a pair $0\leq \bv'\leq\bv$, of dimension vectors,
we choose $I$-graded vector spaces $V,V',$ and $W$ such that
$\dim_I V=\bv,\ \dim_I V'=\bv',$ and $\dim_I W=\bw$.
Thus, we identify  Nakajima's varieties of relevant dimesions
with  corresponding Hamiltonian reductions of the representation
spaces $\Rep(\ddq, V,W)$, resp. $\Rep(\ddq, V',W)$.
Therefore, a choice of $I$-graded vector space
isomorphism $\phi:\ V'\oplus V''\iso V$ clearly
induces a vector space imbedding
$\jmath_\phi:\ \Rep(\ddq, V',W)\to\Rep(\ddq, V,W),\
(\bx',\by',\bi',\bj')\mto(\bx',\by',\bi',\bj')\oplus 0'',$
where $0''\in \Rep(\ddq, V'',W)$ denotes  the zero quadruple.
The latter imbedding induces a morphism
$\M_0(\bv',\bw)\to\M_0(\bv,\bw)$,
of the corresponding categorical quotients.

\begin{rem} Note that the map $\jmath_\phi$ does {\em not}
give rise to any natural morphism
$\M_\th(\bv',\bw)\to\M_\th(\bv,\bw)$ because the stability conditions
involved in the definitions of these spaces are not compatible,
in general.
\erem

We observe that, for any other $I$-graded  vector space
isomorphism $\psi:\ V'\oplus V''\iso V$,
there exists an element $g\in G_\bv$ such
that one has $\jmath_\psi=g\ccirc\jmath_\phi$.
It follows, that the maps $\jmath_\phi$
and $\jmath_\psi$ induce the {\em same}
morphism  $\imath_{\bv',\bv}:\ \M_0(\bv',\bw)\to\M_0(\bv,\bw)$.
Thus, the latter morphism is defined canonically.
Furthermore, according to \cite[Lemma 2.5.3]{Na4},
one has

\begin{lem} For any
dimension vectors $\bv'\leq\bv$,
the canonical morphism $\imath_{\bv',\bv}:
\M_0(\bv',\bw)$ $\to\M_0(\bv,\bw)$
is a closed imbedding.\qed
\end{lem}

Thus, for any $\bv'\leq\bv$, one has
natural inclusions
$\M^\circ_0(\bv',\bw)\into\M_0(\bv',\bw)\into
\M_0(\bv,\bw).$ Here, the first inclusion is an open,
resp. the second inclusion $\imath_{\bv'\bv}$ is a closed, imbedding.
Below, we will identify $\M^\circ_0(\bv',\bw)$ with a subset
of  $\M_0(\bv,\bw)$ via the  composite imbedding and put
\beq{good}\M_0^\go(\bv,\bw)\,:=\,\cup_{0\leq\bv'\leq\bv}\
\M^\circ_0(\bv',\bw).
\eeq

\begin{rem}\label{ADE} \vi If $Q$ is a 
finite Dynkin quiver of type $\mathbf{A},\mathbf{D},\mathbf{E}$,
cf. Example \ref{ADEb}, then,
according to \cite[Remark 3.28]{Na2},  one has
$\M_0(\bv,\bw)=\M_0^\go(\bv,\bw).$

\vii Let $Q$ be  an extended Dynkin quiver and let $\Gamma\sset SL_2(\C)$
be the finite subgroup associated with $Q$ via the McKay correspondence,
cf. \S\ref{mksec}.
Then, one can show that there is a natural decomposition
$$\M_0(\bv,\bw)=\bigcup_{\{\bv'\in
\Z^I_{\geq 0},\
k\geq 0\en|\en\bv'+k\cdot\delta\leq\bv\}}\ 
\M^\circ_0(\bv',\bw)\times\Sym^k(\C^2/\Gamma),
$$
where $\delta$ denotes the minimal imaginary root.

Furthermore, the set $\M_0^\go(\bv,\bw)$
equals the union of pieces in the above decomposition
corresponding to $k=0$.
Thus, in this case, we have $\M_0(\bv,\bw)\neq\M_0^\go(\bv,\bw),$
in general.
\end{rem}

\subsection{A Steinberg type variety}\label{st} 
Let $\bv',\bv\in \Z^I_{\geq 0}$ be a pair of dimension vectors, and
 identify $\M_0(\bv',\bw)$
with a closed subset of $\M_0(\bv+\bv',\bw)$ via the canonical imbedding.
Thus, we get a well defined composite
$\M_\th(\bv,\bw)\to\M_0(\bv,\bw)\into\M_0(\bv,\bw)$,
where the first map is the canonical projective morphism.

\begin{defn}
Given $\th\in\Z^I$ and any pair $\bv,\bv'\in\Z^I$, of dimension vectors,
we define an associated {\em Steinberg variety}
\beq{ZZ}Z_\th(\bv,\bv',\bw)\ :=\ \M_\th(\bv,\bw)\times_{\M_0(\bv+\bv',\bw)}
\M_\th(\bv',\bw)\ \sset \ \M_\th(\bv,\bw)\times \M_\th(\bv',\bw),
\eeq
as a fiber product in the following diagram
\beq{Z}
\xymatrix{
&&Z(\bv,\bv',\bw)
\ar[dll]\ar[drr]&&\\
\M_\th(\bv,\bw)\ar[d]^<>(0.5){\pi}&&&&
\M_\th(\bv',\bw)\ar[d]^<>(0.5){\pi}\\
\M_0(\bv,\bw)\ar@{^{(}->}[rr]_<>(0.5){\imath_{\bv,\bv+\bv'}}&&\M_0(\bv+\bv',\bw)&&
\M_0(\bv',\bw)\ar@{_{(}->}[ll]^<>(0.5){\imath_{\bv',\bv+\bv'}}
}
\eeq
\end{defn}

By definition, the morphisms $\pi$ induce a natural projective morphism
$\pi_Z:\ 
Z_\th(\bv,\bv',\bw)\to\M_0(\bv+\bv',\bw)$.

The Steinberg variety is typically quite singular and has many 
irreducible components. In the special case
where $\bv=\bv'$, the diagonal
$\M_\th(\bv,\bw)\sset \M_\th(\bv,\bw)\times \M_\th(\bv,\bw)$
is one such component, which is smooth provided $\th$ is $\bv$-regular.

Assume now that $\th=\th^+$, and write
$\M(\bv,\bw):=\M_{\th^+}(\bv,\bw)$, resp. $Z(\bv,\bv',\bw)
:=Z_\th(\bv,\bv',\bw).$
Then, 
 $\M(\bv,\bw)$, resp. $\M(\bv',\bw),$
is a smooth symplectic algebraic variety with symplectic 2-form $\om$, resp.
$\om'$. We equip the cartesian product
 $\M(\bv,\bw)\times \M(\bv',\bw)$
 with the
symplectic 2-form $\om+(-\om').$

Next, recall the set introduced in \eqref{good} and put
$$Z^\go(\bv,\bv',\bw):=Z(\bv,\bv',\bw)\
\sminus\
\overline{\pi_Z\inv\big(\M_0(\bv+\bv',\bw)\sminus
\M_0^\go(\bv,\bw)\big)},
$$
where bar stands for the closure.
Thus, $Z^\go(\bv,\bv',\bw)$ is an open subset of $Z(\bv,\bv',\bw)$.

According to \cite{Na2}, Theorem 7.2, one has

\begin{thm}\label{Zthm} \vi Any irreducible component
of   $Z^\go(\bv,\bv',\bw)$ is a
(locally closed) Lagrangian subvariety
of $\M(\bv,\bw)\times \M(\bv',\bw).$

\noindent
\vii The dimension of any irreducible component
of 
$Z(\bv,\bv',\bw)$ is 
$\leq\frac{1}{2}\big(\dim\M(\bv,\bw)+\dim\M(\bv',\bw)\big)$.
\end{thm}

Nakajima also proves, cf. \cite[Corollary 10.11]{Na2}.

\begin{prop} Let $Q$ be either a finite Dynkin
or an extended Dynkin quiver. Then, each irreducible component
of $Z(\bv,\bv',\bw)$ has dimension equal to
$\frac{1}{2}\big(\dim\M(\bv,\bw)+\dim\M(\bv',\bw)\big)$.
\end{prop}
In the case of  finite Dynkin quivers, the result
follows
from
Remark \ref{ADE}.
The extended Dynkin case may be proved using the fact
that a similar result is known to hold for the Jordan quiver,
see \cite[Remark 1.23]{Na6}.\footnote{I am grateful to H. Nakajima for
clarifying this point to me.}

\subsection{Geometric construction of $\wt U(\g)$}\label{Ug_sec}
We keep the assumption and notation
of the previous subsection, in particular, we take $(\la,\th)=(0,\th^+)$.
and write $\M(\bv,\bw):=\M_{0,\th^+}(\bv,\bw),$ etc.
We use simplified notation $\M(\bv,\bw):=\M_{0,\th^+}(\bv,\bw)$,
and $\M_0(\bv,\bw)=\M_{0,0}(\bv,\bw),$ etc.

We introduce the following disconnected varieties
$$ 
M(\bw):=\bigsqcup_{\bv\in\Z^I}
\M(\bv,\bw),\quad M_0(\bw):=\bigsqcup_{\bv\in\Z^I}
\M_0(\bv,\bw),\quad
Z(\bw):=\bigsqcup_{(\bv,\bv')\in\Z^I\times\Z^I} 
 Z(\bw,\bv,\bv').$$

Thus, the morphisms $\pi: \M(\bv,\bw)\to\M_0(\bv,\bw)$
may be assembled together
to give a morphism $M(\bw)\to M_0(\bw)$, and we have
$Z(\bw)=M(\bw)\times_{M_0(\bw)}M(\bw)$. 
Also, we define
$$H_\bw\ :=\  \underset{m\geq 0}\bplus\ 
\left(\prod_{\{(\bv,\bv')\in\Z^I_{\geq 0}\times\Z^I_{\geq 0}\; \big|\;
|\bv-\bv'|\leq m\}}
H_{[0]}\bigl Z(\bw,\bv,\bv')\bigr\right),
$$
where, for any $\bv,\bv'\in\Z^I$, we write
$|\bv-\bv'|:=\sum\ii |v_i-v'_i|$.

Thus, $H_\bw$ is a certain completion of the direct
sum
$\bigoplus_{\bv,\bv'}\ H_{[0]}\bigl Z(\bw,\bv,\bv')\bigr$
whose
elements  are, in general, infinite sums;
at a heuristic level, one has
$H_\bw=H_{[0]}(Z(\bw)).$  It is easy to see that  convolution in
 Borel-Moore homology for various pairs of spaces $Z(\bw,\bv,\bv')$
extends to a well defined operation on $H_\bw$ that
 makes $(H_\bw,\ *)$ an assiciative $\C$-algebra.

We also let $\L(\bv,\bw)=\pi\inv(0)$ be
the zero fiber of the morphism $\pi$, cf. \S\ref{lagr_sec}.
We put 
$$ \L_\bw:=\bigsqcup_{\bv\in\Z^I} \L(\bv,\bw),\en\text{resp.}
\en
L_\bw:=\underset{\bv\in\Z^I}\bplus\
H\top\bigl\L(\bv,\bw)\bigr.
$$
Thus, heuristacally, one has $L_\bw=H\top(\L_\bw)$.
\medskip

Recall next that, associated with the Cartan matrix $C_Q$, of the
quiver $Q$, there is a canonically defined Kac-Moody Lie algebra $\g_Q$,
with   Chevalley generators
$e_i,\ h_i, f_i,\ i\in I,$ see \cite{Ka}. 
 We write ${\mathfrak{h}}$ for the Cartan subalgebra of
$\g_Q$. For each $i\in I$, let $\alpha_i\in{\mathfrak{h}}^*$
denote the corresponding simple root, resp.
$\varpi_i\in{\mathfrak{h}}$ 
denote the corresponding fundamental weight such that $\varpi_i(h_i)=1$ and
$\varpi_i(h_j)=0$ for any $j\neq i$.

Let $U(\g_Q)$ be the universal enveloping algebra of
$\g_Q$. There is a convenient modification of this
algebra  where the Cartan part in  the standard triangular decomposition
of  $U(\g_Q)$ is replaced by the weight lattice.
The resulting algebra
$\wt U(\g_Q)$, called the {\em modified  enveloping algebra},
 was first introduced by Lusztig, cf. \cite{L1}.

One of the main results of Nakajima's theory reads,
see \cite[Theorem 9.4 and \S11]{Na2}
\begin{thm}\label{Ug}
\vi There is a natural algebra homomorphism
$\Psi:\ \wt U(\g_Q)\to  H_\bw$.

\vii The $\wt U(\g_Q)$-action on the
vector space $L_\bw$, induced by the  homomorphism $\Psi$ via
Lemma \ref{top}(ii), makes the latter a simple integrable
$\g_Q$-module with highest weight $\sum\ii w_i\cdot\varpi_i$.
\end{thm}
\begin{rem}
Theorem \ref{lagr} implies that $\L_\bw$ is a (disconnected)
{\em Lagrangian}
subvariety of $M(\bw)$, a disconnected symplectic manifold. It follows that
the fundamental classes of {\em all} irreducible components of
the variety $\L_\bw$ form a natural basis in
the vector space $L_\bw=H\top(\L_\bw)$.
This basis goes, via the identification provided by
Theorem \ref{Ug}(ii),
to a so-called
{\em semicanonical} basis in the corresponding
simple $U(\g_Q)$-module, cf. \cite{L4}.
\end{rem}
 \begin{rem} In the special case where $Q$ is a Dynkin
quiver of type ${\mathbf A}$,  Theorem \ref{Ug}
reduces to an earlier result obtained in \cite{Gi3},
where  the corresponding Steinberg
variety was introduced.

Many interesting interconnections arising
 specifically in the case of quivers of type ${\mathbf A}$
are discussed in \cite{MV}.
\end{rem}

\begin{proof}[Hint on proof of Theorem \ref{Ug}]
The homomorphism $\Psi$, of Theorem
\ref{Ug}(i), is constructed by sending each of the Chevalley generators $e_i,\
h_i, f_i,\ i\in I$
to an appropriate explicit linear combination
of the fundamental classes of some carefully chosen
{\em smooth} irreducible components of the Steinberg
variety $Z(\bw)$.

Specifically, fix $i\in I$ and  let
$\ee^i=(0,\ldots,0,1,0,\ldots,0)\in\Z^I$
denote the $i$-th coordinate vector.
Then, the generator
$h_i$ is sent to a linear combination
of the form $\sum_\bv\ a_\bv\cdot[\M(\bv,\bw)]$,
where $[\M(\bv,\bw)]$ denotes the fundamental
class of the diagonal
$\M(\bv,\bw)\sset\M(\bv,\bw)\times\M(\bv,\bw),$
and $a_\bv\in{\mathbb Q}$ are certain rational coefficients.

The generator $e_i$ is sent to a linear combination
of the form $\sum_\bv\ b_\bv\cdot[Z^i(\bv,\bw)].$
Here $Z^i(\bv,\bw)\sset  \M(\bv,\bw)\times \M(\bv+\ee^i,\bw),$
is a smooth irreducible component 
of the  Steinberg varietiy  $Z(\bv,\bv+\ee^i,\bw),$
and $b_i\in{\mathbb Q}$ are some coefficients.
Similarly, the generator $f_i$ is sent to a linear combination
of the form $\sum_\bv\ c_\bv\cdot[Z^i(\bv-\ee^i,\bw)^\op]$, $c_\bv\in{\mathbb Q}$.
In the last formula,
$Z^i(\bv-\ee^i,\bw)^\op\sset  \M(\bv,\bw)\times \M(\bv-\ee^i,\bw)$
is a subvariety which is obtained from the variety
$Z^i(\bv-\ee^i,\bv)\sset \M(\bv-\ee^i,\bw)\times \M(\bv,\bw),$
involved in the formula for the generator $e_i$,
by the flip-isomorphism
$\M(\bv-\ee^i,\bw)\times \M(\bv,\bw) \cong
\M(\bv,\bw)\times\M(\bv-\ee^i,\bw).$
\end{proof}

The following  result was proved in
\cite[Theorem 10.2]{Na2}

\begin{thm}\label{repthm} Let $x\in \M^\circ_0(\bv',\bw)$
for some $0\leq\bv'\leq\bv,$
and view $x$ as a point in $\M_0(\bv,\bw)$.
Then, one has

\vi The fiber $\M(\bv,\bw)_x=\pi\inv(x)$ is equi-dimensional;

\vii The convolution product makes $H\top(\M(\bv,\bw)_x)$
an $\wt U(\g_Q)$-module. This is an integrable simple
$\wt U(\g_Q)$-module with the highest weight
equal to $\sum\ii\ (w_i\cdot\varpi_i - v'_i\cdot\alpha_i).$
\end{thm}

\begin{rem}
Note that 
part  (ii) of the above theorem
reduces, in the special
case $\bv'=0$, to  Theorem \ref{Ug}(ii).
\erem

In the paper \cite{Na4}, Nakajima proves
analogues of Theorems \ref{Ug} and \ref{repthm}, where
the algebra $\wt U(\g_Q)$ is replaced by
$\wt U_q(\wh \g_Q)$, the (modified) quantized
enveloping algebra of the {\em affinization} of
the Kac-Moody algebra $\g_Q$. Accordingly,
Borel-Moore homology is replaced in \cite{Na4}
by equivariant $K$-theory; in particular,
the algebra $H_\bw$ is replaced
by (a completion of) $K^{G_\bw\times\gm}\big(Z(\bw)\big)$,
the $G_\bw\times\gm$-equivariant $K$-group of the 
Steinberg variety. 

In the special case where $Q$ is a Dynkin
quiver of type ${\mathbf A}_{n-1}$ we have
$\g=\mathfrak{s}\mathfrak{l}_n.$ 
One can use the description of the
corresponding quiver varieties
in terms of partial flag manifolds
provided by Proposition \ref{kpp}.
The results of Nakajima \cite{Na4} reduce, in this case,
to the results obtained  earlier  in \cite{GV}, cf. also \cite{V}
(in these papers, the authors consider the
algebra 
$\wt U_q(\wh {\mathfrak{g}\mathfrak{l}}_n)$
rather than $\wt U_q(\wh {\mathfrak{s}\mathfrak{l}}_n)$,
but the difference is not very essential).

\begin{rem}
The use of equivariant $K$-theory by Nakajima was
strongly motivated
by a similar approach to representations of
affine Hecke algebras that has been known at the
time, see \cite{KL} and \cite{CG}.
\end{rem}
\subsection{Acknowledgements}{\footnotesize{I would like to thank  Michel Brion
for his hard work to make  the Summer School
in Grenoble (2008) successful and for his
kind invitation to participate in  the Summer School.
I am also very grateful to Hiraku Nakajima for explaining to me several
statements and unpublished proofs and for bringing 
reference \cite{Ru} to my attention.

This work was  supported in part  by the NSF   grant  DMS-0601050.}}


\end{document}